\documentclass[11pt,reqno]{amsart}
\usepackage{latexsym}
\usepackage{mathrsfs}
\usepackage{amsmath}
\usepackage{amssymb, amsmath}
\usepackage{graphicx}
\usepackage{float,url,color}
\usepackage[colorlinks=true]{hyperref}
\usepackage{bbm}
\usepackage{enumerate}
\usepackage{amsmath,bm}
\usepackage[numbers,sort&compress]{natbib}
\numberwithin{equation}{section}
\newcommand{\R}{{\mathbb R}}

\newcommand{\be}{\begin{equation}}
\newcommand{\ee}{\end{equation}}

\newcommand{\ben}{\begin{eqnarray*}}
\newcommand{\enn}{\end{eqnarray*}}

\newcommand{\va}{\varepsilon}

\newcommand{\beq}{\begin{equation}}
\newcommand{\eeq}{\end{equation}}
\newcommand{\bea} {\begin{array}{rl}}
\newcommand{\eea} {\end{array}}
\newcommand{\bepa}{\left\{ \begin{array}{l}}
\newcommand{\eepa} {\end{array}\right.}

\newcommand{\Rmnum}[1]{\expandafter\@slowromancap\romannumeral #1@}
\makeatother
\newcommand{\D}{\displaystyle}

\newcommand{\norm}[1]{\left\lVert#1\right\rVert}
\newtheorem{theorem}{\textbf Theorem}[section]
\newtheorem{lemma}{\textbf Lemma}[section]
\newtheorem{rem}{\textbf Remark}[section]

\newtheorem{corollary}[theorem]{Corollary}

\newtheorem{proposition}[theorem]{Proposition}

\def\p{\partial}

\def\pa{\partial}

\renewcommand{\theequation}{\arabic{section}.\arabic{equation}}
\begin{document}

\author{Benoit Perthame}
\address{Sorbonne Universit\'{e}, Universit\'{e} Paris-Diderot, CNRS, INRIA, Laboratoire Jacques-Louis Lions, F-75005 Paris, France}
\email{ Benoit.Perthame@upmc.fr}

\author{Nicolas Vauchelet}
\address{Universit\'e Paris 13, Sorbonne Paris Cit\'e, CNRS UMR 7539, Laboratoire Analyse G\'eom\'etrie et Applications, 93430 Villetaneuse, France}
\email{vauchelet@math.univ-paris13.fr}

\author{Zhian Wang}
\address{Department of Applied Mathematics, Hong Kong Polytechnic University,
Hung Hom, Kowloon, Hong Kong}
\email{mawza@polyu.edu.hk}
\title
[Flux limited Keller-Segel system and kinetic equations]
{The flux limited Keller-Segel system; \\ properties and derivation from kinetic equations}

\begin{abstract}
The flux limited Keller-Segel (FLKS) system is a macroscopic model describing bacteria motion by chemotaxis which takes into account saturation of the velocity. The hyperbolic form and some special parabolic forms have been derived from kinetic equations describing the run and tumble process for bacterial motion. The FLKS model  also has the advantage that traveling pulse solutions exist as observed experimentally. It has attracted the attention of many authors recently.

We  design and prove a general derivation of the FLKS  departing from a kinetic model  under stiffness assumption of the chemotactic response and rescaling the kinetic equation according to this stiffness parameter.
Unlike the classical Keller-Segel system, solutions of the FLKS system do not blow-up in finite or infinite time. Then we investigate the existence of radially symmetric steady state and long time behaviour of this flux limited Keller-Segel system.
\end{abstract}



\date{}
\maketitle

\section{Introduction}

Chemotaxis, the directed movement of an organism in response to a chemical stimulus, is a fundamental cellular process in many important biological processes such as embryonic development \cite{LiMuneoka}, wound healing \cite{Pettet96},  blood vessel formation \cite{ChapLogas, Gamba03}, pattern formation
\cite{Berg91, PainterFish}) and so on.  Well-known examples of biological species experiencing chemotaxis include the slime mold amoebae {\it Dictyostelium discoideum}, the flagellated bacteria {\it Escherichia coli} and {\it Salmonella typhimurium}, and the human endothelial cells \cite{Murray1}. Mathematical models of chemotaxis were derived from either microscopic (individual) or macroscopic (population) perspectives, which have been widely studied in the past four decades. The macroscopic chemotaxis model has been first developed by Keller-Segel in \cite{KS1} to describe the aggregation of cellular slime molds {\it Dictyostelium discoideum} and in \cite{KS3} to describe the wave propagation of bacterial chemotaxis.
Because thresholds on the cell number decide when solutions will undergo smooth dispersion or blow-up in finite time, and because of the interest of related functional analysis, this system has attracted an enormous number of studies (cf. \cite{Perthamebook}).

In this paper, we are interested in the flux-limited Keller-Segel (FLKS) system in the whole space $\R^d$.
Some particular form of such system has already been introduced in \cite{HP09,Chertock}. It describes the evolution of cell density $\rho(t,x)$ and chemical signal concentration $S(t,x)$ at $x\in \R^d$ and time $t>0$, and is based on the physical assumption that the chemotactic flux function is bounded, modeling velocity saturation in large gradient environment. It reads
\begin{equation}\label{ksf}
\begin{cases}
\p_t\rho= D \Delta\rho-{\rm div}(\rho \phi(|\nabla S|)\nabla S), \quad x\in \R^d,\ t>0,
\\[2pt]
\tau \p_t S - \Delta S+ \alpha S = \rho,
\\[2pt]
\rho(0,x)=\rho^0(x) \geq 0, \qquad  \hbox{and } \qquad
S(0,x)=S^0(x) \; \mathrm{if} \ \tau=1.
\end{cases}
\end{equation}
We denote the cell total number $M:=\int_{\R^d} \rho^0(x)\,dx > 0$.
This system is conservative, that is
$$
M=\int_{\R^d} \rho(t,x)\,dx, \qquad \forall t\geq 0.
$$
Compared to the classical Keller-Segel system, the chemotactic response function $\phi\in C^1(\R^+; \R^+)$ depends nonlinearly on the chemical concentration gradient.
We assume flux limitation, that means there is a positive constant $A_\infty$ such that
\beq\label{asF}
\max\limits_{r \in \R^+} \phi(r)= \phi(0), \qquad \max\limits_{r \in \R^+}|r \phi(r)|=A_\infty.
\eeq
These boundedness assumptions on the flux induce that solutions to \eqref{ksf} exist globally in time (see e.g. \cite{HPS,Chertock}), unlike the Keller-Segel system for which finite time blow-up may occur.

The motivation to study the FLKS system \eqref{ksf} comes from its derivation from mesoscopic kinetic model.
The first microscopic/mesoscopic description of chemotaxis model is due to Patlak \cite{patlak} whereby the kinetic
theory was used to express the chemotactic velocity in term of the
average of velocities and run times of individual cells. This approach was essentially boosted by Alt \cite{ALT80} and
developed by Othmer, Dunber and Alt \cite{ODA88} using a
velocity-jump processes which assumes that cells run with some velocity
and at random instants of time they changes velocities (directions)
according to a Poisson process. The advantage of kinetic models over macroscopic models is that details of the run-and-tumble motion at individual scales can be explicitly incorporated into the tumbling  kernel and then passed to macroscopic quantities through bottom-up scaling (cf. \cite{hiloth, Xue1, Xue2, ErbanOth1, ErbanOth2, STY, PTV, DS}), where the rigorous justification of upscaling limits have been studied in many works (see \cite{CMPS, HKS1, HKS3, JV, Liao} and reference therein).
Denoting by $f(t,x,v)$ the cell number density, at time $t$, position $x\in {\mathbb R}^d$ moving with a velocity $v\in V$ (compact set of $\mathbb{R}^d$ with rotational symmetry), the governing evolution equation of this process is described by a kinetic equation reading as:
\begin{eqnarray}\label{kin}
\frac{\partial f}{\partial t} +v\cdot \nabla_x f= \int_{V} \big(T[S](v, v')f(t,x,v')-T[S](v',v)f(t,x,v) \big)dv',
\end{eqnarray}
The tumbling kernel $T[S](v,v')$ describes the frequency of changing  trajectories from velocity $v'$ (anterior) to $v$ (posterior) depending on the chemical concentration $S$ or its gradient. Because cells are able to compare present chemical concentration to previous ones and thus to respond to temporal gradients along their pathways, the tumbling kernel may depend on the pathway (directional derivative)  and takes the form (\cite{DS, PTV})
\begin{equation}\label{T1}
T[S](v,v')=\lambda_0+\sigma \Psi(D_tS), \qquad D_t S=\p_tS+v'\nabla S,
\end{equation}
where $\lambda_0$ denotes a basal meaning tumbling frequency, $\sigma$ accounts for the variation of tumble frequency modulation and $\Psi$ denotes the signal response (sensing) function which is decreasing to express that cells are less likely to tumble when the chemical concentration increases.

The first goal of the present paper is to derive the FLKS system (\ref{ksf}) as the parabolic limit of the kinetic equation (\ref{kin})-(\ref{T1}) and relate the flux limiting function $\phi$ to $\Psi$. In particular, we introduce a new rescaling, related to the stiffness of  signal response, which has been shown to be important to describe the traveling pulses of bacterial chemotaxis observed in the experiment \cite{SCB-PNAS, SCB-PLOS,PLoS2species} and is related to instabilities both of the FLKS system and the kinetic equation~\cite{PY, CPY}. In particular, we wish to go further than the case proposed in \cite{SCB-PLOS}, when the  response function $\Psi$ is bi-valuated step (stiff) function:
$\Psi(Y)=-\mbox{sign} (Y)$,
where the parabolic limit equation of \eqref{kin} is
$$
 {\p_t} \rho= \Delta \rho -{\rm div} (\rho u[S]),\qquad u[S]= J(S_t, |\nabla S|)\frac{\nabla S}{|\nabla S|}.
$$
with $J$ denoting a macroscopic quantity depending on $\nabla S$ and/or $S_t$ (cf. \cite{DS,SCB-PLOS}).
Our method of proof is based on the method of moments and on compactness estimates to treat the nonlinearity.

Our second goal is to study the long time behaviour of solutions to the FLKS system \eqref{ksf} and the existence of stationary radial solutions.
Contrary to the Keller-Segel system for which finite time blow-up of weak solutions is observed, solutions to \eqref{ksf} under assumption \eqref{asF} exist globally in time.
For a study of the long time convergence towards radially symmetric solutions for the Keller-Segel system, which may be computed explicitely, we refer to \cite{CalCar}.
Yet, we do not have an explicit expression of radially symmetric solutions for system \eqref{ksf}.
Then, we prove that when the degradation coefficient is positive ($\alpha>0$), diffusion takes the advantage over attraction.
On the contrary, when the degradation coefficient is disregarded ($\alpha=0$), the total mass of the system, denoted $M$, appears to be an important parameter.
Indeed, when $\alpha=0$, we observe a threshold phenomenon, with a critical mass $M^* = \frac{8\pi}{\phi(0)}$ in dimension $d=2$, for the existence of radial stationary solution.

More precisely, our main results may be summarized as follows :
\begin{itemize}
\item {\it Radial stationary solutions when $\alpha=0$}. (Theorem \ref{th:radial}) \\
For $d>2$, there are no positive radially symmetric steady state solutions to (\ref{ksf}) with finite mass $M>0$. \\
For $d=2$, system (\ref{ksf}) has positive radially symmetric steady state if and only if $M > M^* = \frac{8 \pi}{\phi(0)}$.
\item {\it Long time behaviour in one dimension when $\alpha=0$}. (Corollary \ref{thm1d}) \\
For $d=1$ and $\alpha=0$, for any $M>0$, there exists a unique stationary solution $\bar{\rho}$.
Moreover, denoted by $\rho$ is the solution of the dynamical system \eqref{ksf} with $\tau=0$ and $\alpha=0$. We have
$$
\lim_{t\to+\infty} W_2(\rho(t),\bar{\rho}) = 0,
$$
where $W_2$ denotes the Wasserstein distance of order $2$.

\item {\it Long time behaviour.} (Theorem \ref{thm}) \\
In dimension $d=2$ or $d=3$.
Let $(\rho, S)$ be a solution of (\ref{ksf}) on $\R^d\times[0, \infty)$.
If $\alpha>0$ and $\tau\in\{0,1\}$, or if $\alpha=0$, $\tau=0$ and $M>0$ is small enough, then we have for any $p\in (1,+\infty]$,
\begin{equation*}
\norm{\rho(t)}_{L^p(\mathbb{R}^d)} \leq C t^{-\frac{d}{2}(1-\frac{1}{p})},
\end{equation*}
where $C$ is a nonnegative constant. Notice that this estimate on the time decay is the same as the one for the heat equation.
\end{itemize}

The situation in bounded domain is quite different. Indeed, existence of steady state solutions for $\alpha>0$ on bounded domain with Neumann boundary conditions has been investigated in \cite{Chertock} in one dimension. Based on a bifurcation analysis, they observe spiky solutions when the chemotactic sensibility is large.
See also \cite{LNT,Ni,SWW,Wang} for spiky steady states in chemotaxis models.

The outline of the paper is as follows. In the next section, we derive the flux-limited Keller-Segel model \eqref{ksf} from the kinetic system with the appropriate scaling.
Section \ref{sec:steady} deals with the existence of radially symmetric stationary states in dimension greater than $2$.
The one dimensional case is investigated in section \ref{sec:1D}.
The study of the long time behaviour is performed in section~\ref{sec:asymp} where Theorem~\ref{thm} is proved.
Then we summarize briefly our results in a conclusion and provide open questions related to this work. Finally, an appendix is devoted to some technical lemma useful throughout the paper.

\section{Derivation of FLKS  from kinetic model}

Our approach uses  the stiffness parameter $\va$ and a smoothed stiff response function $\Psi_\va(Y)=\Psi(\frac{Y}{\va})$.
In other words, we consider the following smooth stiff tumbling kernel
$$
T[S](v,v')=\lambda_0+\sigma \Psi(D_tS/\va).
$$
A possible example, as suggested in \cite{SCB-PNAS}, is $\Psi(\frac{Y}{\va})=-\tanh(\frac{Y}{\va})$. Other examples include, for instance, $\Psi(\frac{Y}{\va})=-\frac{Y}{\sqrt{\va^2+Y^2}}$. The case $\va=0$ corresponds to a stepwise stiff response function mentioned in the introduction.
In \cite{SCB-PNAS}, it has been measured that $\frac{\sigma}{\varepsilon}\approx 12$. For convenience, we write $\sigma=\chi\varepsilon$ for some scaling constant $\chi (\approx 12)>0$ and rewrite above tumbling kernel as
\begin{equation}\label{tk}
T[S](v,v')=\lambda_0+\chi \va \Psi(D_tS/\va).
\end{equation}
In this paper, we shall take $\va$ as a scaling parameter and derive the parabolic limit of kinetic models of chemotaxis which turns out to be the FLKS model  (\ref{ksf}) as long as the response function $\Psi$ is bounded.

\subsection{Rescaling of the kinetic equation}

We summarize the condition on the response function $\Psi$ as follows:
\begin{equation} \label{q1}
\Psi \in C^{1}\cap L^\infty(\R), \qquad \Psi'(z) <0 \quad  \forall z \in \R.
\end{equation}
Applying the parabolic scaling $t'=\va^2t, x'=\va x$ into \eqref{kin} with the tumbling kernel \eqref{tk}, and recovering $(t',x')$ by $(t,x)$ for convenience, we get
\begin{equation}\label{k1}
\begin{aligned}
\va^2\frac{\p}{\p t}f_{\va}(t,x,v)+\va v \cdot \nabla_x f_{\va}(t,x,v)= \mathscr{L}_{\va}[S_{\va}](f_{\va})\\
\end{aligned}
\end{equation}
with
$$
\mathscr{L}_{\va}[S](f)=\int_{V} \big(T_{\va}[S](v, v') f'-T_{\va}[S](v',v)f \big)dv', \qquad f':= f(t,x,v'), \; f:= f(t,x,v) ,
$$
and
\begin{align}\label{k3}
T_{\va}[S](v,v')=\lambda_0+\chi \va \Psi_{\va}[S](v,v'), \ \ \Psi_{\va}[S](v',v)=\Psi(\va\p_t S+ v\cdot\nabla_x S).
\end{align}
Since the chemical production and degradation are much slower than the movement (cf. \cite{SCB-PNAS,PLoS2species}), we assume prior to the microscopic scaling that the equation for $S$ is given by
\ben\label{S0}
\tau \p_t S=\Delta S+\varepsilon^2(\rho-\alpha S),
\enn
where $\rho(t,x) = \int_V f(t,x,v) dv$, $\tau=\{0,1\}$ and $\alpha\geq 0$ is a constant denoting chemical decay rate.

After rescaling, we may state the complete problem we are interested in. On one hand,  the equation for the chemical concentration with the parabolic scaling reads as
\be\label{S1}
\begin{cases}
\tau \p_t S_\va - \Delta S_\va+ \alpha S_\va = \rho_\va, \qquad \rho_\va (t,x)=\int_V f_\va (t,x,v) dv ,
\\[5pt]
S_\va(0,x)=S^0(x) \in L_+^1\cap L^\infty(\R^d) \quad \mathrm{for} \ \tau = 1.
\end{cases}
\ee
On the other hand, substituting (\ref{k3}) into (\ref{k1}), we get the final form of the kinetic equation
\begin{equation}\label{k4}
\begin{cases}
\va^2 \D\frac{\p}{\p t}f_{\va}(t,x,v)+\va v\cdot \nabla_xf_{\va}(t,x,v)
=\lambda_0\D\int_{V}(f'_{\va}-f_{\va})dv'+\chi \va\int_{V} (\Psi^*_\va[S_\va] f'_{\va}-\Psi_\va[S_\va]f_{\va})dv',
\\[8pt]
f_{\va}(0,x,v) = f^0(x,v)\in L_+^1\cap L^\infty ( \R^d \times V),
\\[5pt]
\Psi_\va[S_\va]=\Psi(\va\p_t S_\va+ v\cdot\nabla_x S_\va),\ \Psi^*_\va[S_\va]=\Psi(\va\p_t S_\va+ v'\cdot\nabla_x S_\va).
\end{cases}
\end{equation}

\subsection{Well-posedness,  a priori estimates and compactness}
\label{sec:compactness}

The well-posedness and macroscopic limits of kinetic models of chemotaxis where the tumbling kernel depends on the chemical concentration or its spatial derivative have been extensively studied e.g. in \cite{ODA88, othhil,CMPS, HKS1,BCGP} either formally or rigorously, based on the advanced functional analytical tools available for kinetic equations. When the tumbling kernel depends on the pathway derivative $D_tS$, the formal limits have been studied in \cite{othhil,DS} and rigorous justification was given in \cite{ENV} for the two species case. The well-posedness of equations (\ref{k1})-(\ref{S1}) and the limit as $\varepsilon\to 0$ are the direct consequence of the results of \cite{ENV}. For completeness, we present, without proof, the following result
\begin{theorem}[Existence, a priori estimates]
Let $\varepsilon>0$ and assume (\ref{q1}). There exists a unique global solution of (\ref{S1})-(\ref{k4}),
$f_\va \in L^\infty_{\rm loc}([0, \infty); L_+^1 \cap  L^{\infty}(\R^d\times V))$, $ S_\va \in L^\infty_{\rm loc}([0, \infty); L^\infty(\R^d))$. Moreover, there is a constant $C(\lambda_0, \| \Psi\|_\infty)$, independent of $\va$, such that
\begin{equation}\label{est1}
e^{-Ct} \int_{\R^d \times V} f_\va(t)^2 dx dv + \frac{\lambda_0 }{4\va^2} \int_0^t \int_{\R^d \times V\times V} | f_\va'-f_\va |^2 dv dv' dx ds \leq \| f^0 \|^2_2 ,
\end{equation}
\begin{equation}\label{est2}
\| \rho_\va \|_{L^2([0,T] \times \R^d)} + \| J_{\va}  \|_{L^2([0,T] \times \R^d)}  \leq C(T) e^{CT}, \qquad J_{\va}:= \frac{1}{\va} \int_V v f_\va  dv.
\end{equation}
\label{th:apriori}
\end{theorem}
The  flux $J_{\va}$ in \eqref{est2} arises because integration of (\ref{k4}) with respect to $v$ gives
\be\label{k6}
\frac{\p \rho_{\va}}{\p t}+{\rm div}_x J_{\va}=0.
\ee

\begin{lemma}[Strong local compactness on $\nabla S_\va(t,x)$] The signal function $S_\va$ is uniformly bounded and  $\nabla S_\va(t,x)$  is strongly locally compact in $L^1_{\rm loc}((T_1,\infty) \times \R^d)$ for all $T_1 >0$.
\end{lemma}

\proof We use that $\rho_\va $ is bounded in $L_{\rm loc}^\infty((0,\infty); L^1( \R^d) \cap L^2(\R^d))$ from Theorem \ref{th:apriori}.

From usual elliptic or parabolic regularizing effects (see Lemma \ref{g} in Appendix), and using only the above bounds for $\rho_\va$, we conclude that $S_\va$ is bounded in $L_{\rm loc}^\infty((0,\infty); L^p( \R^d))$,  with $p^*:=\frac{d}{d-2}< p \leq  \frac{2d}{d-4}$ ($p=\infty$ in dimensions $d=2, \; 3$, $p <\infty$ in dimension 4).

 Next, the equation on $ \p_t S_\va$ reads
$$
\tau \p_t ( \p_t S_\va) - \Delta  (\p_t S_\va)+ \alpha ( \p_t S_\va) =  \p_t \rho_\va = - {\rm div} J_{\va}, \quad J_{\va} \in_{\rm bdd} L^2([0,T] \times \R^d)
$$
and gives $ \p_t S_\va \in L^r((T_1,T); L^{2} (\R^d))$, for any $r<2$ any  $0<T_1<T$ thanks to the direct estimates on the heat kernel (see \eqref{hk} in Appendix with $p=1$, $l=0$, $k=1$).

Next, we notice by a similar argument,  that $\nabla S_\va$ is bounded in $L^\infty((T_1,T); L^2( \R^d))$ for $0<T_1<T$.  Compactness in $x$ for $\nabla S_\va$ also follows from the convolution formula.  Finally, we write
$$
\tau \p_t ( \p_t \nabla S_\va) - \Delta  (\p_t  \nabla S_\va)+ \alpha ( \p_t  \nabla S_\va)  = - \nabla{\rm div} J_{\va},
$$
and thus we conclude that $ \p_t  \nabla S_\va \in_{\rm bdd} L^2((T_1,T) \times \R^d)$, for all $0<T_1<T$ , thus providing time compactness.
\qed

As a conclusion, we may extract a subsequence $\{\varepsilon(n)\}_{n \geq 1}$ such that $\nabla S_{\va(n)} \to \nabla S_0$ locally  in all spaces $L_{\rm loc}^r((0,\infty); L^q( \R^d))$, $1\leq r < \infty$, $\frac{d}{d-1}<q \leq 2p^*$.  Notice that the bounds above tell us that, as~$n \to \infty$,
$$
\Psi_{\va(n)}[S_{\va(n)}]=\Psi({\va(n)} \p_t S_{\va(n)}+ v\cdot\nabla_x S_{\va(n)}) \to \Psi (v\cdot\nabla_x S_0).
$$

\subsection{The convergence result}

As a consequence of the a priori estimates in Theorem~\ref{th:apriori} and the discussion in section \ref{sec:compactness}, we may also extract subsequences (still denoted by $\{\va(n)\}$) such that, weakly in $L^2([0,T] \times \R^d)$ for all $T >0$, as~$n \to \infty$, we have
\begin{equation}\label{conv}
f_{\va(n)}  \to  f_0(t,x,v)=  \rho_0(t,x) F(v), \qquad  J_{\va(n)}(t,x):= \frac{1}{\va(n)} \int_V v f_{\va(n)}  dv \to J_0(t,x)
\end{equation}
where $F(v)$ is a uniform  distribution on $V$:
\begin{equation}\label{F}
F(v)=\frac{\mathbbm{1}_{\{v\in V\}}}{|V|}.
\end{equation}
With the symmetry assumption of $V$, it satisfies $\int_V vF(v)dv=0$, and $\int_V F(v)dv=1. $ In the limit we infer from (\ref{k6}) that
\be\label{k6L}
\frac{\p \rho_{0}}{\p t}+\nabla_x \cdot J_{0}=0.
\ee

The flux $J_0$ can be identified and we are going to show in the next subsection the following
\begin{theorem} [Derivation of the FLKS system]
Assuming  (\ref{q1}), the above limit $(\rho_0, S_0)$ satisfies the FLKS system~(\ref{ksf}) with initial condition $(\int_V f^0(x,v) dv, S^0)$ and
$$
D=\frac{1}{\lambda_0 |V|^2} \int_V v\otimes v dv, \qquad  \phi(u) = - \frac{1}{\lambda_0 |V| u} \int_V v_1 \Psi(v_1 u) dv >0 \; \mbox{ for } u\geq 0,
$$
where $v_1$ is the first component of the vector field $v$.
\label{th:asuympt}
\end{theorem}

Notice that $\phi(0)$ is well defined by continuity and $\phi(0) =-\frac{\Psi'(0)}{\lambda_0|V|} \int_V v_1^2  dv $.

\subsection{Asymptotic analysis}

In order to complete the proof of Theorem~\ref{th:asuympt}, we proceed to find the flux term $J_{0}$ in \eqref{k6L}. Multiplying (\ref{k4}) by $v$ and integrating, we get
\begin{equation*}\label{k4n}
\begin{aligned}
&\va \frac{\p}{\p t}J_{\va}(t,x)+ \nabla_x \cdot \int_V v\otimes v f_{\va}dv
\\
&= \frac{\lambda_0}{\va} \int_{V}v\int_V(f'_{\va}-f_{\va})dv'dv+ \chi \int_{V} v\int_{V} (\Psi^*_\va[S_\va] f'_{\va}-\Psi_\va[S_\va]f_{\va})dv'dv
\\
&= - \lambda_0 |V| J_{\va} - \chi |V|  \int_{V} v\Psi_\va[S_\va]f_{\va} dv
\end{aligned}
\end{equation*}
using the definition of $J_{\va}$  in \eqref{est2}, the definition of $\Psi^*_\va$ in \eqref{k4}, and the symmetry of $V$.

We may pass to the weak limit and find, based on the above mentioned strong compactness for $S_\va$ and its derivatives as well as (\ref{conv}) and (\ref{F}), that
$$
\int_V v\otimes v dv  \;  \nabla_x \rho_0 = - \lambda_0 |V| J_0 - \chi |V|  \rho_0  \int_{V} v\Psi (v\cdot \nabla S_0) dv.
$$
In other words, we have identified the flux term
\be\label{k8}
J_0=-\nabla_x \rho_0\frac{1}{\lambda_0 |V|^2}\int_Vv\otimes v dv-\frac{\chi}{\lambda_0|V|}\rho_0\int_V v \Psi (v\cdot \nabla S_0)dv.
\ee
Using (\ref{k8}), the leading order terms of (\ref{k6}) and (\ref{S1}) lead to the following drift-diffusion equations:
\begin{eqnarray}\label{ks}
\begin{cases}
\p_t \rho_0={\rm div}(D\nabla \rho_0-\chi\rho_0 u[S]),\\
\tau \p_tS_0=\Delta S_0+\rho_0-\alpha S_0,
\end{cases}
\end{eqnarray}
where
\begin{eqnarray}\label{add}
\begin{aligned}
D=\frac{1}{\lambda_0 |V|^2}\int_V v\otimes v dv,\quad
u[S]=-\frac{1}{\lambda_0|V|}\int_V v\Psi(v\cdot \nabla S_0)dv.
\end{aligned}
\end{eqnarray}
By rotational symmetry of $V$, $u[S]$ is proportional to $\nabla S$ and hence yields the expression of $\phi(u)$ in Theorem \ref{th:asuympt}.
Due to the hypothesis (H) on $\Psi$, the drift velocity term $u[S]$ is uniformly bounded in time $t$ and space $x$. This is the main feature of the macroscopic limit model resulting from the stiff response postulated in the kinetic models.
\qed
\subsection{Example}
We consider a specific form of signal response function $\Psi$ as follows
\begin{equation*}\label{phi}
\Psi({Y}/{\va})=-\frac{Y}{\sqrt{\va^2+Y^2}} \ \mbox{ or } \ \Psi(z)=-\frac{z}{\sqrt{1+z^2}}
\end{equation*}
and derive an explicit flux-limited Keller-Segel system. When $\varepsilon=0$,  $\Psi(Y/\va)=-\mathrm{sign}(Y)$ which is a sign function reflecting the stepwise stiff response. However, as $\varepsilon>0$, $\Psi(Y/\va)$ is smooth and $\Psi'(0)=-\frac{1}{\varepsilon}$.

By substitution, we have from (\ref{k3}) that
$$
\Psi_{\va}[S](v',v)=-\frac{\va \p_t S+v\cdot \nabla S}{\sqrt{1+(\va \p_tS+v\cdot\nabla S)^2}}.
$$
Then the limit equations of (\ref{k1})-(\ref{S1}) read as (see \eqref{ks}-\eqref{add})
\begin{eqnarray}\label{ksn}
\begin{cases}
\p_t \rho={\rm div}(D\nabla \rho-\rho \phi(|\nabla S|)\nabla S),\\
\tau \p_tS=\Delta S+\rho-\alpha S,\\
\end{cases}
\end{eqnarray}
where we have recovered $(\rho_0, S_0)$ by $(\rho, S)$ for brevity and,
by rotational symmetry of V,
\begin{eqnarray}\label{addn}
\begin{aligned}
D=\frac{1}{\lambda_0 |V|^2}\int_V v\otimes v dv,\quad
\phi(|\nabla S|)=\frac{\chi}{\lambda_0|V|}\int_V \frac{v_1^2}{\sqrt{1+v_1^2 |\nabla S|^2}}dv,
\end{aligned}
\end{eqnarray}
where $v_1$ is the first component of $v$ colinear to $\nabla S$: $v_1=v\cdot \frac{\nabla S}{|\nabla S|}.$
Clearly both $\phi$ and $\phi(|\nabla S|)\nabla S$ are bounded for all $\nabla S$, which implies that the chemotactic (drift) velocity is limited. The system (\ref{ksn}) with (\ref{addn}) gives a specific example of the FLKS system (\ref{ksf}).


\subsection{Global existence for the macroscopic system}

Finally, we state the existence result for system \eqref{ksf} under the assumption \eqref{asF}.
A specific example of function $\phi$ satisfying this set of assumptions has been given in \eqref{addn}.
Under the assumption \eqref{asF}, the chemotactic (or drift) velocity term $\phi(|\nabla S|)\nabla S$ is bounded and hence the global existence of classical solutions of (\ref{ksf}) can be directly obtained.
\begin{theorem}[Global existence]\label{Global}
Let $0\leq (\rho^0, S^0) \in (W^{1, p}(\R^d))^2$ with $p>d$ and $\alpha \geq 0$.
Let $\phi\in C^1(\R^+;\R^+)$ such that \eqref{asF} holds. Then the Cauchy problem (\ref{ksf}) has a unique solution $(\rho, S) \in C([0,\infty)\times \R^d) \times C^2( (0,\infty)\times \R^d)$ such that
$$
\forall\, t>0, \quad \|\rho(t,\cdot)\|_{L^\infty(\R^d)}\leq C, \quad d \geq 2,
$$
where $C>0$ is a constant independent of $t$. Moreover cell mass is conserved: $\|\rho(t)\|_{L^1(\R^d)}=\|\rho^0\|_{L^1(\R^d)}=M$.
\end{theorem}
\begin{proof}
The proof consists of two steps. The first step is the local existence of solutions which can be readily obtained by the standard fixed point theorem (cf. \cite{Am1, Am2}). The second step is to derive the {\it a priori} $L^\infty$ bound of $u$ in order to extend local solutions to global ones. This can be achieved by the method of  Nash iterations as it is well described in \cite[Lemma 1]{HPS}. Although the procedure therein was shown for bounded domain with Neumann boundary conditions, the estimates directly carry over to the whole space $\R^d$.
\end{proof}

\section{Radial steady states in dimension $d\geq 2$}\label{sec:steady}

Since it is proved in section \ref{sec:asymp} that when $\alpha>0$ diffusion takes the advantage over attraction implying the time decay towards zero of the solutions to system \eqref{ksf}, we are only interested in the case $\alpha=0$.
The stationary problem for system (\ref{ksf}) is non-trivial due to the nonlinearity.
Below we explore a simpler case: existence of radial symmetric stationary solutions.
The stationary system of \eqref{ksf} when $\alpha=0$ written in radial coordinates for $d\geq 2$ reads
\beq \bepa
\D  - \frac{1}{r^{d-1}} (r^{d-1} S(r)')' = \rho(r), \qquad r >0,
\\[10pt]
\D \frac{1}{r^{d-1}} \big[r^{d-1} \big(-\rho(r)' + \rho(r) S(r)'  \phi(|S(r)'|) \big) \big]'=0,
\\[10pt]
S'(0)=0, \qquad  \rho'(0)=0.
\eepa
\label{rss:eq1}
 \eeq
Notice that there is another relation, at infinity, expressing that the mass is given by
\beq
\frac{M}{|\mathbb{S}^{d-1}|}=  \int_0^{+\infty} r^{d-1} \rho(r) dr =-  \int_0^{+\infty} (r^{d-1} S(r)')' dr= -  \D \lim_{r\to \infty} r^{d-1} S(r)'.
\label{rss:eqm}
 \eeq

We are going to prove the following result.
\begin{theorem}\label{th:radial}
There are no positive radially symmetric steady state solutions with finite mass to system (\ref{ksf}) with $\alpha=0$ in dimension $d>2$.
In dimension 2 (i.e. $d=2$), system (\ref{ksf}) with $\alpha=0$ has radially symmetric steady states if and only if $M > \frac{8 \pi}{\phi(0)}$.
\end{theorem}

\begin{proof}
We use the unknown $v(r) = -  r^{d-1}S'(r) \geq 0$ in order to carry out the analysis.
The equation on $\rho$ in \eqref{rss:eq1} now reads $- \rho'(r) + \rho(r) S'(r)  \phi(|S'(r)|)  =0$, and we obtain
\beq \bepa \label{eq2}
v'= r^{d-1}  \rho(r) \geq 0,
\\[10pt]
\rho' = - \rho \; \frac{v(r)}{r^{d-1}} \phi \left(\frac{v(r)}{r^{d-1}} \right),\\[10pt]
v(0)=0,\ \rho(0)=a>0.
\eepa
\eeq
From (\ref{asF}) and the second equation of \eqref{eq2}, we get $\rho'\geq -\rho A_\infty$ and hence $\rho(r) \geq a e^{-A_\infty r}>0$ for all $r \in [0,\infty)$. Furthermore from the first equation of \eqref{eq2}, we know that $v$ is non-decreasing and has a limit as $r\to \infty$ which determines the total mass according to~\eqref{rss:eqm}. Hence for finite mass $M$, $v(r)$ has a finite limit and thus for $r$ large enough, say $r \geq r_0$ for some $r_0>0$, we have
\beq \label{master}
\big(\ln \rho \big)' =-\frac{v(r)}{r^{d-1}} \phi \left(\frac{v(r)}{r^{d-1}} \right) \geq - \frac{b}{r^{d-1}},
\eeq
for some positive constant $b>0$.

If $d>2$, integrating (\ref{master}) yields
$$\rho(r)\geq C e^{\frac{b}{d-2}\frac{1}{r^{d-2}}}.$$
This is incompatible with finite mass $M$ in \eqref{rss:eqm}.

In dimension $d=2$, we recover a phenomenon similar to the multiple solutions for the critical mass in the Keller-Segel system but explicit solutions are not available. The system (\ref{eq2}) reduces to
\begin{equation}\label{eqr:d2}
\left\{\begin{array}{ll}
r v'' = v' \left(1-v \phi(\frac{v}{r})\right), &\qquad r>0,  \\
v(0)=v'(0)=0.
\end{array}\right.
\end{equation}
By the boundary conditions in \eqref{eq2}, we see that positive solutions behave as $v(r) \approx \frac{a }{2}r^2$ for $r \approx 0$ and some constant $a>0$. Then, the proof of Theorem~\ref{th:radial} is a consequence of Lemma \ref{lem:nec} and Proposition \ref{prop:M} below.
Lemma \ref{lem:nec} states that a necessary condition of existence of radial solution is $M>\frac{8\pi}{\phi(0)}$.
Proposition \ref{prop:M} shows that for any finite mass $M$ larger than the critical mass $\frac{8\pi}{\phi(0)}$, there exist radial solutions with mass $M$ to system \eqref{ksf} with $\alpha=0$.
\end{proof}

\begin{lemma}\label{lem:nec}
Let $v$ be a positive solution to \eqref{eqr:d2}.  Then $v$ is increasing.
If $v$ is bounded, then
$$
\displaystyle \lim_{r\to +\infty} v(r) > \frac{4}{\phi(0)}.
$$
\end{lemma}
\begin{proof}
We split the proof into three steps:
\begin{enumerate}
\item From the behaviour near $r=0$, we know that $v'(r) >0$ for $r>0$ small enough. If we had $v'(r_0)=0$ for some $r_0>0$, then the unique solution of \eqref{eqr:d2} is $v(r)= v(r_0)$ which is a contradiction. Therefore $v'(r) >0$ for all $r>0$.
\item Since $\phi(\cdot) \leq \phi(0)$ from \eqref{asF}, we deduce from \eqref{eqr:d2} that, for all $r>0$,
$$
 v'(1-v\phi(0)) \leq r v'' .
$$
This inequality may be rewritten as
\begin{equation}\label{ineqvprim1}
 v'(2-v\phi(0)) \leq   (r v')' .
\end{equation}
Integrating (\ref{ineqvprim1}) from $0$ to $r$ and using boundary conditions in (\ref{eqr:d2}), we deduce that
$$
2 v - \frac{\phi(0)}{2} v^2 \leq  r v'  .
$$
This inequality implies that
$$
\lim_{r\to \infty} v (r) \geq  \frac{4}{\phi(0)}
$$
because if it were smaller, we would have $v'(r) > c/r$ for some $c>0$ and $\frac 1 r$ is not integrable.

As a consequence, we know that as $r\to \infty$, $rv''\leq -3 v' $ and thus,  for some nonnegative constant $C$, it holds that
\beq
r^3 v'(r) \leq C, \ \ \mathrm{as} \ \ t \to \infty.
\label{rs:estdiff}
\eeq
\item  We may go further and write the first equation of (\ref{eqr:d2}) as
$$
 (r v')' =  v'(2-v\phi(0)) + v v' [\phi(0)- \phi(\frac vr )] .
$$
Integrating it from $0$ to $r$, we have
$$
r v' = 2 v -  \phi(0) \frac{v^2}{2} + Q(r), \qquad Q(r) = \int_0^r  v(s) v' (s) [\phi(0)- \phi(\frac{v(s)}{s} )]ds >0.
$$
Therefore as $r \to \infty$, using \eqref {rs:estdiff}, there holds that
$$
\phi(0) \frac{v_\infty ^2}{2} - 2 v_\infty =Q(\infty) >0.
$$
It implies that $v_\infty>\frac{4}{\phi(0)}$.
\end{enumerate}
\end{proof}

\begin{proposition}\label{prop:M}
Let the function $\phi$ satisfies \eqref{asF}.
Then for any $b> \frac{4}{\phi(0)}$, there exists a solution $v$ to \eqref{eqr:d2} such that $\lim_{r\to + \infty} v(r) = b$.

\end{proposition}

\begin{proof}
We want to prove that for any $b>\frac{4}{\phi(0)}$, there exists $a>0$ such that the solution
to \eqref{eqr:d2} verifying $v''(0)=a$ and $\lim_{r\to + \infty} v(r) = b$.
We first simplify the problem by introducing the change of variable $y=\frac{r^2}{r^2+1}\in [0,1)$.
Setting $u(y)=v(r)$, we deduce, from straightforward computations, that $u$ is a solution to the system
\begin{equation}\label{equy}
\left\{\begin{array}{ll}
u'' = \frac{2 u'}{1-y} \left(1-\frac{u}{4y} \phi\Big(\sqrt{\frac{1-y}{y}} u\Big)\right), &\qquad y\in (0,1),  \\[3mm]
u(0)=0, \quad u'(0)=\frac{a}{2}.
\end{array}\right.
\end{equation}
We are left to use a shooting method to show there is a number $a>0$ such that (\ref{equy}) has a solution satisfying $u(1)=\lim_{y\to 1} u(y)=b$ for any $b> \frac{4}{\phi(0)}$.

\begin{itemize}
\item By definition of $u$ and thanks to the above results, we have that $u\geq 0$ and $u'\geq 0$ on $[0,1)$.
\item Since $\phi(\cdot)\leq \phi(0)$, we deduce from \eqref{equy} that
$$
(y(1-y)u')' \geq -\frac{u u'}{2} \phi(0) + u'.
$$
After integration we obtain
$$
y(1-y) u' \geq - \frac{u^2}{4} \phi(0) + u.
$$
Thus, when $u(y)\leq \frac{4}{\phi(0)}$, we have
$$
\frac{u'(y)}{u(y)-u^2(y) \phi(0)/4} \geq \frac{1}{y(1-y)}.
$$
Upon integration, we find a positive constant $\lambda>0$ such that for all $y\in(0,1)$ and $u(y)\leq \frac{4}{\phi(0)}$, we have
$$
u(y) \geq \frac{4\lambda y}{1-y+\lambda \phi(0) y} \ \to \ \frac{4}{\phi(0)} \ \mathrm{as} \ y \to 1.
$$
Thus, by continuity and the fact $u$ is increasing, we have that $u(1)\geq \frac{4}{\phi(0)}$.

\item Let us prove that: For any $b>\frac{4}{\phi(0)}$, there exists a number $a>0$ small enough such that
the solution to \eqref{equy} satisfies $u(1)\leq b$. \\
In the vicinity of $0$, we have $u(y)\sim \frac{a}{2} y$. Then for $a>0$ small enough, there exists $y_0\in(0,1)$ such that $u(y_0)=a y_0$.
(Indeed, if it is not true, we will have $u(y)\leq a y$ on $(0,1)$, which is not possible for $a<\frac{4}{\phi(0)}$ since $u(1)\geq \frac{4}{\phi(0)}$).
The function $y\mapsto \sqrt{\frac{1-y}{y}} u(y)$ being bounded on $(0,1)$, let us denote
$\phi_m = \min_{y\in(0,1)} \phi\Big(\sqrt{\frac{1-y}{y}} u(y)\Big)$.
By the same token as above, we deduce from \eqref{equy} that
$$
(y(1-y)u')' \leq -\frac{u u'}{2} \phi_m + u'.
$$
Integrating above inequality over $(0,y)$ gives
\begin{equation}\label{ieq}
y(1-y)u'\leq -\frac{u^2}{2}\phi_m+u.
\end{equation}
On one hand, integrating (\ref{ieq}) from $\frac 12$ to $y$, we get
$$
u(y) \leq \frac{2 u(\frac 12) y}{2(1-y)(1-\frac{\phi_m}{4}u(\frac 12)) + \frac{\phi_m}{2}u(\frac 12) y}\leq C_m y.
$$
On the other hand, integrating (\ref{ieq}) between $y_0$ and $y$, we obtain
$$
u(y) \leq \frac{u(y_0) y/y_0}{\frac{1-y}{1-y_0}(1-\frac{u(y_0)}{4}\phi_m) + \frac{\phi_m}{4}u(y_0) \frac{y}{y_0}}
= \frac{ a y(1-y_0)}{1-y + \frac{\phi_m}{4} a (y-y_0)}.
$$
Then,
$$
\sqrt{\frac{1-y}{y}} u(y) \leq \xi=
\min\left(C_m\sqrt{y(1-y)},\frac{ a \sqrt{y(1-y)}(1-y_0)}{1-y + \frac{\phi_m}{4} a (y-y_0)}\right).
$$
It is clear that $\xi \to 0$ as $a \to 0$ for all $y\in [0,1]$. Let $\varepsilon>0$ small. Then by the continuity of the function $\phi$, for $a>0$ small enough, we can deduce from the above estimate that
$\phi\big(\sqrt{\frac{1-y}{y}} u(y)\big) \geq \phi(0)- \varepsilon$.
Then, we can redo the same estimate as above, replacing $\phi_m$ by $\phi(0)- \varepsilon$, we arrive at
$$
u(y) \leq
\frac{ a y(1-y_0)}{1-y + \frac{\phi(0)-\varepsilon}{4} a (y-y_0)},
$$
which implies by taking $y=1$,
$$
u(1) \leq b=: \frac{4}{\phi(0)-\varepsilon}.
$$

\item Let us prove that $\lim_{a\to +\infty} u(1)=+\infty$.

By the second assumption on $\phi$ in \eqref{asF}, we know that for any $u>0$ and $y\in(0,1)$,
$$
\phi\Big(\sqrt{\frac{1-y}{y}} u \Big) \leq \frac{A_\infty}{u}\sqrt{\frac{y}{1-y}}.
$$
Since $u'\geq 0$, we get from \eqref{equy} that
$$
u'' \geq \frac{2 u'}{1-y} \left(1 - \frac{A_\infty}{4\sqrt{y(1-y)}}\right).
$$
We may integrate this inequality between $0$ and $y$ for $y\in (0,1)$, and get
$$
\ln u'(y) - \ln (\frac{a}{2}) \geq -2 \ln (1-y)
-\frac{A_\infty}{2}\int_0^y \frac{dz}{\sqrt{z}(1-z)^{3/2}}.
$$
We deduce that for any $y\in[0,\frac 12]$,
$$
u'(y) \geq \frac{a}{2}
\exp\left(-\frac{A_\infty}{2}\int_0^{1/2} \frac{dz}{\sqrt{z}(1-z)^{3/2}}\right).
$$
Hence the integration of last inequality from $0$ to $\frac{1}{2}$ yields
$$
u\Big(\frac{1}{2}\Big) - u(0) \geq
\frac{a}{4}\exp\left(-\frac{A_\infty}{2}\int_0^{1/2} \frac{dz}{\sqrt{z}(1-z)^{3/2}}\right)\underset{a\to +\infty}{\longrightarrow} +\infty.
$$
Since $u$ is nondecreasing, we have
$u(1)\geq u\Big(\frac{1}{2}\Big)$. This implies that $\lim_{a\to +\infty} u(1)=+\infty$.

\item We are now in a position to conclude the proof.
The function $a\mapsto u(1)$ is continuous. We have proved that $\underset{a\to 0}{\lim\inf}\  u(1) = \frac{4}{\phi(0)}$ and $\underset{a\to +\infty}{\lim} u(1) = +\infty$.
Thus for any $b>\frac{4}{\phi(0)}$ there exists $a>0$ such that the solution to \eqref{equy} verifies $u(1) = b$. This completes the proof.
\end{itemize}
\end{proof}


\section{One dimensional case}\label{sec:1D}

In one dimension, we can improve the above results and show the existence and uniqueness of a steady state for any finite $M>0$ and the convergence (in Wasserstein distance) of the solution $\rho(t)$ towards this unique steady state as $t\to +\infty$.
Let us consider system (\ref{ksf}) when $\alpha=0$ and $\tau=0$ in one dimension:
\begin{align}
&\pa_t \rho - \pa_{xx} \rho + \pa_x (\rho \phi(|\pa_x S|)\pa_x S) = 0,
\label{eq1D:n}  \\
&- \pa_{xx} S = \rho,
\label{eq1D:S}  \\
& \rho(t=0) = \rho^0 \in L^1_+(\R), \qquad \|\rho^0\|_{L^1} = M >0.
\label{eq1D:n0}
\end{align}
We assume that $\phi\in C^1(\R^+;\R^+)$ and $\phi$ verifies assumption \eqref{asF}.

In order to reduce the problem, we define $u=-\pa_x S$, such that $\rho=\pa_x u$.
We notice that
$-\rho u\phi(|u|) = -\pa_x \Phi(u)$, where $\Phi$ is an antiderivative of $x\mapsto x\phi(|x|)$. Remark that $\Phi$ is even and nondecreasing on $\R^+$.
As a consequence, system \eqref{eq1D:n}--\eqref{eq1D:S} reduces to
\begin{equation}\label{eq1D:u}
\pa_t u - \pa_{xx} u - \pa_x \Phi(u) = 0, \qquad u(t=0,x)=u^0(x) := \int_{-\infty}^x \rho^0(y)\,dy - \frac{M}{2}.
\end{equation}
We assume moreover that $|u^0|-\frac{M}{2} \in L^1(\R)$.

As it is now standard for parabolic equation, we may prove easily the following existence result:
\begin{lemma}\label{lemu}
Let $M>0$ be given and let us assume that $\rho^0\in L^1_+(\R)$.
Then, there exists a unique solution $u$
to \eqref{eq1D:u} which satisfies
$$
0 \leq \pa_x u \in L^\infty((0,+\infty);L^1(\R)), \quad u(0)=0,
\quad \lim_{x \to \pm \infty} u(t,x)=\pm \frac{M}{2}.
$$
If we assume moreover that $|u^0|-\frac{M}{2} \in L^1(\R)$, then we also have
that $|u|-\frac{M}{2} \in L^\infty((0,+\infty);L^1(\R))$.
\end{lemma}

\subsection{Steady state}

We now investigate the existence of steady state for the system in one dimension:
\begin{lemma}\label{lemubar}
Let $M>0$ be fixed.
There exists a unique steady state $\bar{u}$ for \eqref{eq1D:u}
which satisfies $\bar{u}(0)=0$, $0\leq \pa_x\bar{u} \in L^1(\R)$ and
$\|\pa_x\bar{u}\|_{L^1(\R)} = M$.
\end{lemma}
\begin{proof}
The steady states are given by
$$
\pa_{xx} \bar{u} = -\pa_x \Phi(\bar{u}).
$$
Since $\Phi$ is defined up to a constant, this problem is invariant by translation, thus we may fix $\bar{u}(0)=0$.
Integrating the latter equation, taking into account the boundary condition at infinity, the steady state is a solution to the Cauchy problem
\begin{equation}\label{EDOstat}
\bar{u}'(x) = \Phi(\frac{M}{2})-\Phi(\bar{u}), \qquad \bar{u}(0) = 0.
\end{equation}
We recall that $\Phi$ is even.
The constant functions $\pm\frac{M}{2}$ being clearly solutions to this differential equation, but not satisfying the boundary condition, we have by uniqueness that the function $\bar{u}$ never reaches the values $\pm \frac{M}{2}$. Then, $|\bar{u}|<\frac{M}{2}$. It implies, by the assumptions \eqref{asF} on $\phi$ that $\bar{u}'>0$.
Thus $\lim_{x\to \pm \infty} \bar{u}$ exists and is finite.
Since from \eqref{EDOstat}, we have $\bar{u}(x) = \int_0^x (\Phi(\frac{M}{2})-\Phi(\bar{u}(y)))\,dy$, we deduce that $y\mapsto \Phi(\frac{M}{2}) - \Phi(\bar{u}(y))$ is integrable on $\R$. Necessarily $\lim_{y\to \pm \infty} (\Phi(\frac{M}{2}) - \Phi(\bar{u}(y))) = 0$.
Since $\Phi$ is even and nondecreasing on $\R^+$, we deduce that $\lim_{x\to \pm \infty} \bar{u}(x) = \pm \frac{M}{2}$ and $\int_{\R} \bar{u}'(x)\,dx = M$.
Let us denote
\begin{equation}\label{defA}
A \mbox{ is an antiderivative of } u\mapsto \frac{1}{\Phi(\frac{M}{2})-\Phi(u)}.
\end{equation}
It is an increasing function, thus it is invertible.
Therefore the solution to the Cauchy problem \eqref{EDOstat} is given by
$$
\bar{u}(x) = A^{-1}(x + A(0)).
$$
\end{proof}

\subsection{Asymptotic behaviour}

\begin{proposition}\label{limu}
Let $u$ and $\bar{u}$ be as in Lemma \ref{lemu} and \ref{lemubar} respectively.
For $A$ defined as in \eqref{defA}, we introduce
$$
E(t) := \int_{\R} \int_{\bar{u}}^u (A(v)-A(\bar{u}))\,dvdx \geq 0.
$$
Then we have the estimate
\begin{equation}\label{estim}
\frac{d}{dt} E(t) =
- \int_\R \big(\Phi(\frac{M}{2})-\Phi(u)\big) \big|\pa_x (A(u)-A(\bar{u}))\big|^2 \,dx
\leq 0.
\end{equation}

If we assume moreover that the initial data $u_0$ is such that $E(0)<\infty$, then
$$
\lim_{t\to + \infty} E(t) = 0.
$$
\end{proposition}
\begin{proof}
We first notice that by definition of $E$, since $A$ is a nondecreasing function, we have $E(t)\geq 0$. Then $E$ may be seen as an entropy. Next we complete our proof in a series of steps.
\smallskip

{\it Step 1 (Entropy dissipation).}
We may rewrite equation \eqref{eq1D:u} as
$$
\pa_t u - \pa_x \left(\pa_x u + \Phi(u)-\Phi(\frac{M}{2})\right) = 0,
$$
which, thanks to the definition of the steady state $\bar{u}$, rewrites as
$$
\pa_t u - \pa_x \left(\Big(\Phi(u)-\Phi(\frac{M}{2})\Big) \pa_x (A(\bar{u})-A(u))\right) = 0.
$$
Multiplying by $A(\bar{u})-A(u)$ and integrating over $\R$, we obtain
$$
\int_\R \pa_t u (A(\bar{u})-A(u)) \,dx + \int_\R \Big(\Phi(u)-\Phi(\frac{M}{2})\Big) |\pa_x (A(\bar{u})-A(u))|^2 \,dx = 0.
$$
We deduce that \eqref{estim} holds.
\smallskip

{\it Step 2 (Compactness argument).} \\
Integrating \eqref{estim} in time, we deduce that for any $t>0$,
$$
E(t) + \int_0^t \int_\R \big(\Phi(\frac{M}{2})-\Phi(u)\big) \big|\pa_x (A(u)-A(\bar{u}))\big|^2 \,dxds = E(0) < +\infty.
$$
In particular, it implies that
$\int_\R \big(\Phi(\frac{M}{2})-\Phi(u)\big) \big|\pa_x (A(u)-A(\bar{u}))\big|^2 \,dx\in L^1(\R_+)$.

Thus there exists a sequence $t_j\to +\infty$ such that
$$
D(t_j) := \int_\R \big(\Phi(\frac{M}{2})-\Phi(u(t_j))\big) \big|\pa_x (A(u(t_j))-A(\bar{u}))\big|^2 \,dx \underset{j\to +\infty}{\longrightarrow} 0.
$$
Expanding and using a Young inequality, we get
$$
\big|\pa_x (A(u(t_j)))-1\big|^2 = |\pa_x (A(u(t_j)))|^2 - 2 \pa_x (A(u(t_j)))
+1 \geq \frac 12 |\pa_x (A(u(t_j)))|^2 - 1.
$$
Thus there exists a nonnegative constant such that
\begin{align*}
\int_{\R} \big(\Phi(\frac{M}{2})-\Phi(u(t_j))\big) \big|\pa_x (A(u(t_j)))\big|^2\,dx
&\leq C + 2\int_{\R} \big(\Phi(\frac{M}{2})-\Phi(u(t_j))\big)\,dx  \\
&\leq C + 2 \|\phi\|_\infty \int_{\R} \left|\frac{M}{2} - |u(t_j)|\right|\,dx.
\end{align*}
From Lemma \ref{lemu}, the last term of the right hand side is bounded.
By definition of $A$, we also get
$$
\int_{\R} \big(\Phi(\frac{M}{2})-\Phi(u(t_j))\big) \big|\pa_x (A(u(t_j)))\big|^2\,dx
 = \int_{\R} \frac{|\pa_x u(t_j)|^2}{\Phi(\frac{M}{2})-\Phi(u(t_j))}\,dx
 = \int_{\R} |\pa_x B(u(t_j))|^2\,dx,
$$
where $B$ is an antiderivative of $u\mapsto \frac{1}{\sqrt{\Phi(\frac{M}{2})-\Phi(u)}}$, then $B$ is increasing and invertible.

We deduce from the above computation that the sequence $\{\pa_x B(u(t_j))\}_j$ is uniformly bounded in $L^2(\R)$.
Therefore, we can extract a subsequence, still denoted $\{B(u(t_j))\}_j$ converging in $L^2_{loc}(\R)$ and a.e. towards a limit denoted $\bar{B}$ as $j\to +\infty$.
As a consequence, $\{u(t_j)\}_j$ converges a.e. towards $u_\infty = B^{-1}(\bar{B})$ as $j\to +\infty$.
\smallskip

{\it Step 3 (Identification of the limit).} \\
To identify this limit, we first notice that since $u(t_j,0)=0$ for any $j\in\mathbb{N}$, then $u_\infty(0)=0$.
Moreover, for any $h\in L^2(\R)$, we have, by a Cauchy-Schwarz inequality
$$
\int_\R \left(\pa_x B(u(t_j)) - \sqrt{\Phi(\frac{M}{2})-\Phi(u(t_j))}\right)h \,dx \leq
\left(\int_{\R} h^2\,dx\right)^{1/2}
D(t_j)^{1/2}  \underset{j\to +\infty}{\longrightarrow} 0.
$$
It implies that for any $h\in L^2(\R)$,
$$
\int_\R \left(\pa_x B(u_\infty) - \sqrt{\Phi(\frac{M}{2})-\Phi(u_\infty)}\right)h(x)\,dx = 0.
$$
We deduce that $u_\infty$ satisfies the problem \eqref{EDOstat}. By uniqueness, we have $u_\infty=\bar{u}$.

Finally, for any regular function $\chi$ compactly supported,
we have,
$$
\int_{\R} \int_{\bar{u}}^{u(t_j)} (A(\bar{u})-A(v)) \chi(x) \,dvdx \underset{j\to +\infty}{\longrightarrow} 0.
$$
Then, choosing $\chi\in C^\infty$ such that $\chi(x)=1$ on $[-\frac 12,\frac 12]$ and $\chi(x)=0$ for $|x|\geq 1$,
\begin{align*}
E(t_j)& \leq \int_{\R} \int_{\bar{u}}^{u(t_j)} (A(\bar{u})-A(v)) \chi\big(\frac{x}{R}\big) \,dvdx + \int_{\R} \int_{\bar{u}}^{u(t_j)} (A(\bar{u})-A(v))\left(1- \chi\big(\frac{x}{R}\big)\right) \,dvdx \\
      & \leq \int_{\R} \int_{\bar{u}}^{u(t_j)} (A(\bar{u})-A(v)) \chi\big(\frac{x}{R}\big) \,dvdx + \|1-\chi\|_\infty \int_{\R\setminus [-R,R]}\int_{\bar{u}}^{u(t_j)} (A(\bar{u})-A(v)) \,dvdx.
\end{align*}
The second term of the right hand side goes to $0$ as $R\to +\infty$, the first term converges to $0$ as $j\to +\infty$.
We deduce that $E(t_j) \to 0$ as $j\to+\infty$.
Since $E$ is decreasing, we conclude that $\lim_{t\to +\infty} E(t) = \lim_{j\to +\infty} E(t_j) = 0$.

\end{proof}

\begin{corollary}\label{thm1d}
Let $M>0$ and $(\rho,S)$ be a solution to system \eqref{eq1D:n}--\eqref{eq1D:n0} with $\phi$ satisfying \eqref{asF} and with an initial data $\rho^0\in L^1(\R)$ such that
$x\mapsto \left|\int_{-\infty}^x \rho^0(y)dy - \frac{M}{2}\right|-\frac{M}{2}$ belongs to $L^1(\R)$.

Let $\bar{\rho}=\pa_x \bar{u}$ where $\pa_x \bar{u}$ is defined in Lemma \ref{lemubar}.

Then we have
$$
\lim_{t\to + \infty} W_2(\rho(t),\bar{\rho}) = 0,
$$
where $W_2$ denotes the Wasserstein distance of second order.
\end{corollary}

\begin{rem}
The assumption on the initial data is automatically satisfied if $\rho^0$ is compactly supported, since then $x\mapsto \left|\int_{-\infty}^x \rho^0(y)dy - \frac{M}{2}\right|-\frac{M}{2}$ is also compactly supported.
\end{rem}

\begin{proof}
This is a direct consequence of Proposition \ref{limu}. Indeed, we have $\rho=\pa_x u$ and $\bar{\rho}=\pa_x\bar{u}$ and
\begin{align*}
E(t) & = \int_\R \int_{\bar{u}}^u (A(v) - A(\bar{u}))\,dvdx
 = \int_\R \int_{\bar{u}}^u \int_v^{\bar{u}} \frac{dw}{\Phi(\frac{M}{2})-\Phi(w)}dvdx  \\
& \geq \frac{1}{\Phi(\frac{M}{2})-\Phi(0)} \int_{\R} \int_{\bar{u}}^u \int_v^{\bar{u}} \,dwdvdx = \int_\R\frac{(u-\bar{u})^2}{2(\Phi(\frac{M}{2})-\Phi(0))}\,dx.
\end{align*}
We conclude by using the fact that
$W_2(\rho,\bar{\rho}) = \| u-\bar{u} \|_{L^2(\R)},$
and $\lim_{t\to + \infty} E(t)=0$.
\end{proof}

\section{Long time asymptotics}\label{sec:asymp}

Now, we investigate the asymptotic dynamics in long time of solutions to the flux-limited Keller-Segel system (\ref{ksf}) in physical dimension $d=2$ or $d=3$. We show that for the chemical decay rate $\alpha>0$, then $\rho(t,x) \to 0$ as $t \to \infty$. While when the chemical decay is ignored ($\alpha=0$), we obtain the convergence $\rho(t,x) \to 0$ as $t \to \infty$ under the assumption that the cell mass $M=\int_{\R^d} \rho^0(x)dx$ is small.
\begin{theorem}\label{thm}
Let $d=2,3$ and  $(\rho, S)$ be a solution of (\ref{ksf}) on $\R^d\times[0, \infty)$. Then for any $M>0$ when $\alpha>0$ and $\tau=\{0,1\}$ or small $M>0$ when $\alpha=0$ and $\tau=0$
it holds that
\begin{equation*}
\begin{aligned}
\norm{\rho(t)}_{L^p(\mathbb{R}^d)} \leq C t^{-\frac d2(1-\frac 1p)},
\end{aligned}
\end{equation*}
where  $1<p\leq \infty$ and $C>0$ is a constant independent of $t$.
\end{theorem}
In both cases, the estimate of order of convergence in time of the norm of $\rho$ towards $0$ is the same as the one for the heat equation.
We will use the following notations. The heat kernel is denoted by $G$:
$$
G(t,x)=\frac{1}{(4\pi t)^{d/2}}\exp\left(-\frac{|x|^2}{4t}\right), \quad x \in \R^d, \ t>0.
$$
It generates a semi-group whose operator is denoted by $e^{t \Delta}$, i.e. $e^{t \Delta} f= G(t) * f$.

\subsection{Asymptotics with chemical decay ($\alpha>0$)}
First we remark that in the case $\alpha>0$, as a direct consequence of Lemma \ref{g} in Appendix along with the fact that $\rho \in L^q(\R^d)$ for $q \in [1, \infty]$ (see Theorem \ref{Global}), we have
\begin{equation}\label{grads}
\forall\, t>0, \quad
\|\nabla S(t)\|_{L^p(\R^d)}< \infty, \mbox{ for all } p \in [1, \infty],
\end{equation}
which is crucial to prove the following result.
\begin{lemma}\label{rho1}
Let $\beta$ be a constant with $0< \beta\leq 1/2$. Then for any $1<p<\frac{d}{d-2}$ $(d\geq 2)$, there is a constant $C>0$ such that, for $t\geq 1$,
$$\|\rho(t)\|_{L^p(\R^d)} \leq C t^{\beta-\frac{d}{2}(1-\frac{1}{p})}.$$
\end{lemma}
\begin{proof}
With a stretching transformation by borrowing an idea from \cite{NSU}, we introduce
$$
\rho_\lambda(t,x)=\lambda^d \rho( \lambda^2 t,\lambda x), \ S_\lambda(t,x)=S(\lambda^2t,\lambda x), \quad x\in \R^d, \ t>0.
$$
We find from (\ref{ksf}) that $(\rho_\lambda, S_\lambda)$ satisfies
\begin{eqnarray*} \label{ksfn}
\left\{\begin{array}{l}
\p_t\rho_\lambda=\Delta \rho_\lambda -{\rm div} [\rho_\lambda \phi(\frac{1}{\lambda}|\nabla S_\lambda|)\nabla S_\lambda ], \ \ x\in \R^d, \ t > 0,\\[2mm]
\tau \p_t S_\lambda=\Delta S_\lambda-\alpha \lambda^2 S_\lambda+\lambda^{2-d}\rho_\lambda.
 \end{array} \right.
\end{eqnarray*}
It can also be easily checked that
\begin{eqnarray}\label{in}
\begin{aligned}
&\|\nabla \rho_\lambda(t)\|_{L^p(\R^d)}=\lambda^{1+d-\frac{d}{p}}\|\nabla\rho(\lambda^2t)\|_{L^p(\R^d)},\\
&\|\nabla S_\lambda(t)\|_{L^p(\R^d)}=\lambda^{1-\frac{d}{p}}\|\nabla S(\lambda^2t)\|_{L^p(\R^d)},\\
&\|\rho_\lambda(t)\|_{L^p(\R^d)}=\lambda^{d-\frac{d}{p}}\|\rho(\lambda^2t)\|_{L^p(\R^d)}.
\end{aligned}
\end{eqnarray}
By Duhamel's principle, we write $\rho_\lambda$ from (\ref{ksf}) as
$$
\rho_\lambda(t)=e^{t\Delta}\rho^0_{\lambda}-\int_0^t \nabla \cdot e^{(t-s)\Delta}\Big(\rho_\lambda \phi\big(\frac{1}{\lambda}|\nabla S_\lambda|\big)\nabla S_\lambda\Big)(s)ds,
$$
where $\rho^0_{\lambda}(x)=\lambda^d\rho^0(\lambda x)$. For given constant $\beta$ with $0\leq \beta<1/2$, we define
$$
r=\begin{cases}\frac{d}{2(\frac{1}{2}-\beta)}, &\ \mathrm{if}\ 0<\beta<\frac{1}{2},  \\ \infty, &\ \mathrm{if}\ \beta=\frac{1}{2}.
\end{cases}
$$
Using the properties of $e^{t\Delta}$ in Lemma \ref{Gheat} in Appendix and the fact that $|\phi\big(\frac{1}{\lambda}|\nabla S_\lambda|)|\leq A_0$ (cf. \eqref{asF}), we find constants $1<q\leq p \leq \infty$ with $\frac{1}{r}+\frac{1}{p}=\frac{1}{q}$ such that
\begin{eqnarray}\label{inn1}
\begin{aligned}
\|\rho_\lambda(t)\|_{L^p(\R^d)}&\leq  \|e^{t\Delta}\rho^0_{\lambda}\|_{L^p(\R^d)}\\
& \qquad +C\int_0^t(t-s)^{-\frac{d}{2}(\frac{1}{q}-\frac{1}{p})-\frac{1}{2}}
\Big\|\rho_\lambda(s) \phi\Big(\frac{1}{\lambda}|\nabla S_\lambda(s)|\Big)\nabla S_\lambda(s)\Big\|_{L^q(\R^d)}ds\\
&\leq \|e^{t\Delta}\rho^0_{\lambda}\|_{L^p(\R^d)}+C\int_0^t(t-s)^{-\frac{d}{2}(\frac{1}{q}-\frac{1}{p})-\frac{1}{2}}\|\rho_\lambda(s) \nabla S_\lambda(s)\|_{L^q(\R^d)}ds\\
&\leq C t^{-\frac{d}{2}(1-\frac{1}{p})} \|\rho^0_{\lambda}\|_{L^1(\R^d)}\\
&\qquad +C \int_0^t(t-s)^{-\frac{d}{2}(\frac{1}{q}-\frac{1}{p})-\frac{1}{2}}\|\rho_\lambda(s)\|_{L^p(\R^d)}\| \nabla S_\lambda(s)|\|_{L^r(\R^d)}ds, \\
\end{aligned}
\end{eqnarray}
where Lemma \ref{Gheat} in Appendix and the H\"{o}lder inequality have been used for the last inequality.
Since $\rho^0_{\lambda}(x)=\lambda^d\rho^0(\lambda x)$, it is easy to verify that $\|\rho^0_{\lambda}\|_{L^1(\R^d)}=\|\rho^0\|_{L^1(\R^d)}=M$.
From \eqref{in} and (\ref{grads}), it follows that $\|\nabla S_\lambda\|_{L^r(\R^d)}=\lambda^{1-\frac{d}{r}}\|\nabla S\|_{L^r(\R^d)} \leq C \lambda^{1-\frac{d}{r}}$ for some constant $C>0$. Then we update (\ref{inn1}) as
\begin{eqnarray*}\label{inn2}
\begin{aligned}
\|\rho_\lambda(t)\|_{L^p(\R^d)}
&\leq C M t^{-\frac{d}{2}(1-\frac{1}{p})} + C \lambda^{1-\frac{d}{r}}\int_0^t(t-s)^{\beta-1}\|\rho_\lambda(s)\|_{L^p(\R^d)}ds\\
&\leq C M t^{-\frac{d}{2}(1-\frac{1}{p})} + C \lambda^{2\beta}\int_0^t(t-s)^{\beta-1}\|\rho_\lambda(s)\|_{L^p(\R^d)}ds,
\end{aligned}
\end{eqnarray*}
which, along with the singular Gronwall's inequality (see Lemma \ref{gronwall}), yields
\begin{eqnarray}\label{inn3}
\begin{aligned}
\|\rho_\lambda(t)\|_{L^p(\R^d)}
\leq C t^{-\frac{d}{2}(1-\frac{1}{p})} + C \lambda^{2\beta}\int_0^t(t-s)^{\beta-1}s^{-\frac{d}{2}(1-\frac{1}{p})}ds.
\end{aligned}
\end{eqnarray}
Let $1<p<\frac{d}{d-2}$. Then $1-\frac{d}{2}(1-\frac{1}{p})>0$ and
\begin{eqnarray}\label{inn3*}
\begin{aligned}
&\int_0^t(t-s)^{\beta-1}s^{-\frac{d}{2}(1-\frac{1}{p})}ds\\
&=\int_0^{t/2}(t-s)^{\beta-1}s^{-\frac{d}{2}(1-\frac{1}{p})}ds+\int_{t/2}^t(t-s)^{\beta-1}s^{-\frac{d}{2}(1-\frac{1}{p})}ds\\
&\leq \Big(\frac{t}{2}\Big)^{\beta-1}\int_0^{t/2}s^{-\frac{d}{2}(1-\frac{1}{p})}ds
+\Big(\frac{t}{2}\Big)^{-\frac{d}{2}(1-\frac{1}{p})}\int_{t/2}^t(t-s)^{\beta-1}ds\\
&\leq C t^{\beta-\frac{d}{2}(1-\frac{1}{p})},
\end{aligned}
\end{eqnarray}
where $C$ is a positive constant depending on $d, p$ and $\beta$.

Let $0<t\leq 1$. Then $t^\beta\leq 1$ and as a result of \eqref{inn3}-\eqref{inn3*} it holds that
$$\|\rho_\lambda(t)\|_{L^p(\R^d)} \leq C (1+\lambda^{2\beta})t^{-\frac{d}{2}(1-\frac{1}{p})}.$$
By the third inequality in (\ref{in}) with $t=1$, we have
$$\|\rho(\lambda^2)\|_{L^p(\R^d)} \leq C \lambda^{-d+\frac{d}{p}}\|\rho_\lambda(1)\|_{L^p(\R^d)} \leq C \lambda^{-d+\frac{d}{p}}(1+\lambda^{2\beta}).$$
Since $\lambda$ is arbitrary, we get by letting $\lambda=\sqrt{t}$ that
\begin{equation*}
\|\rho(t)\|_{L^p(\R^d)} \leq C t^{\beta-\frac{d}{2}(1-\frac{1}{p})} \ \ \mathrm{for} \ \ t\geq 1,
\end{equation*}
which completes the proof.
\end{proof}

Then we investigate the time decay of $\|\rho(t)\|_{L^\infty(\R^d)}$.
\begin{lemma}\label{rho2}
Let $d=2,3$. For any $t\geq 1$, there is a constant $C>0$ such that the solution component $\rho(t,x)$ satisfies for $1<p\leq \infty$
$$\|\rho(t)\|_{L^p(\R^d)} \leq C t^{-\frac{d}{2}(1-\frac{1}{p})}.$$
\end{lemma}

\begin{proof}
We shall first prove $\|\rho(t)\|_{L^\infty(\R^d)} \leq C t^{-\frac{d}{2}}$. From the first equation of (\ref{ksf}), we can write $\rho$ as
\begin{equation}\label{u-I}
\begin{split}
  \rho(t)=&e^{t\Delta}\rho^0-\int_0^t\nabla \cdot e^{(t-s)\Delta}(\rho\phi(|\nabla S|)\nabla S)(s)ds\\
  =&e^{t\Delta}\rho^0-\int_0^{\frac{t}{2}}\nabla \cdot e^{(t-s)\Delta}(\rho\phi(|\nabla S|)\nabla S)(s)ds-\int_{\frac{t}{2}}^t\nabla \cdot e^{(t-s)\Delta}(\rho\phi(|\nabla S|)\nabla S)(s)ds\\
  =&I_0+I_1+I_2.
\end{split}
\end{equation}
Next we estimate $I_i (i=0,1,2)$. First by Lemma \ref{Gheat} in Appendix and the fact $ \|\rho^0\|_{L^1(\R^d)}=M$, we get
\begin{equation}\label{I0}
  \|I_0\|_{L^\infty(\R^d)}\leq C  t^{-\frac{d}{2}}.
\end{equation}
By Lemma \ref{Gheat} in Appendix and the fact that $\phi(|\nabla S|)$ is bounded (cf \eqref{asF}), we have for any $p>1$ that
\begin{equation*}
\begin{split}
  \|I_1\|_{L^\infty(\R^d)}\leq C&\int_0^{\frac{t}{2}}(t-s)^{-\frac{d}{2}-\frac{1}{2}}\|\rho\phi(|\nabla S|)\nabla S\|_{L^1}ds\\
  \leq C&\int_0^{\frac{t}{2}}(t-s)^{-\frac{d}{2}-\frac{1}{2}}\|\rho |\nabla S|\|_{L^1}ds\\
  \leq C&\int_0^{\frac{t}{2}}(t-s)^{-\frac{d}{2}-\frac{1}{2}}\|\rho\|_{L^p}\|\nabla S\|_{L^\frac{p}{p-1}}ds.
\end{split}
\end{equation*}
Now we choose $p$ such that $\frac{d}{d-1}<p<\frac{d}{d-2}$. Then from \eqref{grads}, we know $\|\nabla S\|_{L^\frac{p}{p-1}}<\infty$, and furthermore using Lemma \ref{rho1} it holds that
\begin{equation*}
\begin{split}
  \|I_1\|_{L^\infty}\leq C&\int_0^{\frac{t}{2}}(t-s)^{-\frac{d}{2}-\frac{1}{2}}\|\rho \|_{L^p}ds\\
  \leq C&t^{-\frac{d}{2}-\frac{1}{2}}\int_0^{\frac{t}{2}}(1+s)^{-\frac{d}{2}(1-\frac{1}{p})+\beta}ds\\
  \leq C &t^{-\frac{d}{2}-\frac{d}{2}(\frac{d-1}{d}-\frac{1}{p})+\beta}.
\end{split}
\end{equation*}
Let $0<\beta<\frac{d}{2}(\frac{d}{d-1}-\frac{1}{p})$ and define $\eta_1=\frac{d}{2}(\frac{d-1}{d}-\frac{1}{p})-\beta$.
Then $\eta_1>0$, hence
\begin{equation}\label{I1}
  \|I_1\|_{L^\infty(\R^d)}\leq C t^{-\frac{d}{2}-\eta_1},\ \ \ 0<\eta_1<\frac{d}{2}\left(\frac{d-1}{d}-\frac{1}{p}\right).
\end{equation}
Next, we estimate $\|I_2\|_{L^\infty(\R^d)}$. First let us pick $q$ such that $q>d$ and let $1<p<d/(d-2)$ if $d\geq 3$ and $p=q$ if $d=2$. Note that here $p$ has nothing to do with the $p$ used in the estimate for $I_1$. Then by interpolation and the boundedness of $\|\rho(t)\|_{L^\infty(\R^d)}$, we have from Lemma \ref{rho1} that, for $t\geq 1$,
\begin{equation*}
  \|\rho\|_{L^q(\R^d)}\leq\|\rho\|_{L^p(\R^d)}^{\frac{p}{q}}\|\rho\|_{L^\infty(\R^d)}^{1-\frac{p}{q}}\leq Ct^{-\frac{p}{q}[\frac{d}{2}(1-\frac{1}{p})-\beta]},
\end{equation*}
which, along with the results in Lemma \ref{g}, yields that
\begin{equation}\label{ee}
  \|\nabla S\|_{L^\infty(\R^d)}\leq Ct^{-\frac{p}{q}[\frac{d}{2}(1-\frac{1}{p})-\beta]}, \quad q>d.
\end{equation}
These two latter inequalities yield that
\begin{equation*}
  \|\rho(s)\nabla S(s)\|_{L^q(\R^d)}\leq C\|\rho(s)\|_{L^q(\R^d)}\|\nabla S(s)\|_{L^\infty(\R^d)}\leq C s^{-\frac{2p}{q}[\frac{d}{2}(1-\frac{1}{p})-\beta]}.
\end{equation*}
Then we have by Lemma \ref{Gheat},
\begin{equation*}
\begin{split}
  \|I_2\|_{L^\infty(\R^d)}\leq C&\int_{\frac{t}{2}}^{t}(t-s)^{-\frac{d}{2q}-\frac{1}{2}}\|\rho(s)\nabla S(s)\|_{L^q(\R^d)}ds\\
  \leq C&\int_{\frac{t}{2}}^{t}(t-s)^{-\frac{d}{2q}-\frac{1}{2}}s^{-\frac{2p}{q}[\frac{d}{2}(1-\frac{1}{p})-\beta]}ds\\
  \leq C&t^{-\frac{2p}{q}[\frac{d}{2}(1-\frac{1}{p})-\beta]}\int_{\frac{t}{2}}^{t}(t-s)^{-\frac{d}{2q}-\frac{1}{2}}ds\\
  \leq C&t^{-\frac{d}{2q}+\frac{1}{2}-\frac{dp}{q}(1-\frac{1}{p})+\frac{2\beta p}{q}}\\
  \leq C&t^{-\frac{d}{2}-l(p,q)+\frac{2\beta p}{q}},
\end{split}
\end{equation*}
where $l(p,q)=-\frac{d}{2}-\frac{d}{2q}-\frac{1}{2}+\frac{dp}{q}$.

Consider $d=2,3$. For $d=2$, from $p=q$, we see that $l(p,q)=\frac{1}{2}-\frac{1}{q}>0$ by choosing $q>2$. When $d=3$, $l(p,q)=\frac{3}{q}(p-\frac{1}{2})-2$ and  $l(p,q)>0 \Leftrightarrow q<\frac{3}{2}(p-\frac{1}{2})$.
Now by choosing $p$ such that $\frac{d}{2}+1<p<\frac{d}{d-2}$ (i.e. $\frac{5}{2}<p<3$ for $d=3$), we can verify that  $\frac{3}{2}(p-\frac{1}{2})>3$.
Hence we can choose $3<q<\frac{3}{2}(p-\frac{1}{2})$ such that $l(p,q)>0$.
Hence for such $p$ and $q$, we have that $l(p,q)-\frac{2p}{q}\beta>0$ for $\beta>0$ sufficiently small.
This produces
\begin{equation}\label{I2}
  \|I_2\|_{L^\infty(\R^d)}\leq C t^{-\frac{d}{2}-\eta_2},\ \ \ 0<\eta_2<l(p,q)-\frac{2p}{q}\beta.
\end{equation}
Then substituting (\ref{I0}), (\ref{I1}) and (\ref{I2}) into (\ref{u-I}), we arrive at
\begin{equation*}
\|\rho(t)\|_{L^\infty(\R^d)}\leq C t^{-\frac{d}{2}}, \quad \mbox{ for } t\geq 1.
\end{equation*}
Thus the interpolation inequality gives
$\norm{\rho}_{L^p(\mathbb{R}^d)}\le \norm{\rho}_{L^1(\mathbb{R}^d)}^{\frac 1p} \norm{\rho}_{L^\infty(\mathbb{R}^d)}^{1-\frac 1p} \leq C t^{-\frac{d}{2}(1-\frac 1p)}$ by noting that $\norm{\rho}_{L^1(\mathbb{R}^d)}=M$.
\end{proof}

\subsection{Asymptotics without chemical decay ($\alpha=0$)}
In this section, we shall explore the asymptotic behavior of solutions of (\ref{ksf}) with $\alpha=0$ as time $t \to \infty$. From the estimates in Appendix \eqref{grad2}-(\ref{grad3}), we see that in this case the estimate (\ref{ee}) does not hold for both $\tau=0$ and $\tau=1$,  and hence the approach in the preceding section can not be used. In particular, in the case $\tau=1$, from (\ref{grad3}), we see that even the basic inequality (\ref{grads}) does not hold.   However in the case $\tau=0$, we can derive the asymptotic behavior of solutions and so we now consider the following parabolic-elliptic model
\begin{equation}\label{system}
\begin{cases}
\rho_t=\Delta\rho-{\rm div}(\rho \phi(|\nabla S|)\nabla S), \ t>0, \ x\in \R^d\\
-\Delta S=\rho,\\
\rho(0,x)=\rho^0(x),
\end{cases}
\end{equation}
where $\phi$ satisfies the condition (\ref{asF}).

\subsubsection{Two dimensional case ($d=2$)}
To derive the asymptotic decay of solutions in two dimensions, we shall employ the so-called \textit{Method of Trap} (e.g. see \cite{BDEF}), which essentially can assert the following result~:
\begin{lemma}\label{trap}
Let $\varphi(t)$ be a continuous function on $[0,\infty)$ with $\varphi(0)=0$ satisfying the following inequality for some constants $m, \beta, \theta>0$,
$$
\varphi(t)\leq c_0 m+c_1 m^\beta (\varphi(t))^\theta
$$
where $c_0, c_1$ are positive constants. If  $\theta>1$ and $0<\beta<1$, then
$$\varphi(t) <(c_1 \theta m^\beta)^{\frac{1}{1-\theta}}$$
provided that $m<m_0$, where $m_0=[\frac{1}{c_0}\frac{\theta-1}{\theta}(c_1\theta)^{\frac{1}{1-\theta}}]^{\frac{1-\theta}{\beta+\theta-1}}$.
\end{lemma}
\begin{proof}
This result can be proved directly by applying the \textit{Method of Trap} in \cite{BDEF}. So we omit the details.
\end{proof}

We will make use of Lemma \ref{trap} to prove our results. To this end,  we first use Duhamel's principle to write the solution $\rho(t,x)$ as
\begin{equation*}
\begin{split}
\rho(t_0+t,x)=
\int_{\mathbb{R}^2}G(t,x-y)\rho(t_0,y)dy
-\int_0^t\nabla G(t-s)*(\rho \phi(|\nabla S|)\nabla S)(t_0+s)ds.
\end{split}
\end{equation*}
where $G(t,x)=\frac{1}{4\pi t}\exp(-\frac{|x|^2}{4t})$ is the heat kernel on $\R^2$. Then by Young's convolution inequality and Lemma \ref{Gheat} in Appendix, we have for $\varrho>2$
\begin{equation}\label{02}
\begin{split}
\norm{\rho(t_0+t)}_{L^\infty(\mathbb{R}^2)}\le & \frac{1}{4\pi t}\norm{\rho(t_0)}_{L^1(\mathbb{R}^2)}\\
&+ A_0 C(\varrho) \int_0^t (t-s)^{-\frac{1}{\varrho}-\frac{1}{2}}\norm{(\rho \nabla S)(t_0+s)}_{L^\varrho(\mathbb{R}^2)}ds.
\end{split}
\end{equation}
where the boundedness of $\phi$ has been used (cf. \eqref{asF}).

On one hand, the H\"older inequality gives us that
\begin{equation}\label{04}
\begin{split}
\norm{(\rho \nabla S)(t_0+s)}_{L^\varrho(\mathbb{R}^2)}
&\leq \norm{\rho(t_0+s)}_{L^p(\mathbb{R}^2)}\norm{\nabla S(t_0+s)}_{L^q(\mathbb{R}^2)},
\end{split}
\end{equation}
where $\frac 1\varrho=\frac 1p+\frac 1q$, $q>2$. On the other hand,
by estimate (\ref{HLIN}) in Appendix, there exists a constant $C(r)$ depending on $r$ such that, for any $s>0$,
\begin{equation}\label{05}
\begin{split}
\norm{\nabla S(t_0+s)}_{L^q(\mathbb{R}^2)}\le C(r)\norm{\rho(t_0+s)}_{L^r(\mathbb{R}^2)},
\end{split}
\end{equation}
where $r=\frac{2q}{2+q}$ (or $\frac 1r=\frac 1q+\frac 12$), $1<r<2$. Substituting \eqref{05} into \eqref{04} gives
\begin{equation}\label{ine}
\begin{split}
\norm{(\rho \nabla S)(t_0+s)}_{L^\varrho(\mathbb{R}^2)}
\leq C(r) \norm{\rho(t_0+s)}_{L^p(\mathbb{R}^2)}\norm{\rho(t_0+s)}_{L^r(\mathbb{R}^2)},
\end{split}
\end{equation}
with $\frac 1\varrho=\frac 1p+\frac 1r-\frac 12$. Then we apply the interpolation inequality
\begin{equation*}
\begin{split}
\norm{\rho}_{L^\gamma(\mathbb{R}^2)}\le \norm{\rho}_{L^1(\mathbb{R}^2)}^{\frac 1\gamma} \norm{\rho}_{L^\infty(\mathbb{R}^2)}^{1-\frac 1\gamma}\le M^{\frac 1\gamma} \norm{\rho}_{L^\infty(\mathbb{R}^2)}^{1-\frac 1\gamma}, \quad 1<\gamma<\infty,
\end{split}
\end{equation*}
to (\ref{ine}) and get
\begin{equation}\label{06}
\begin{split}
\norm{(\rho \nabla S)(t_0+s)}_{L^\varrho(\mathbb{R}^2)}
 \leq C(r) M^{\frac 1p+\frac 1r} \norm{\rho(t_0+s)}_{L^\infty(\mathbb{R}^2)}^{2-\frac 1p-\frac 1r}.
 \end{split}
\end{equation}
Substituting  \eqref{06}  into \eqref{02}, we find a constant $C_0=C_0(\varrho, A_0, r)>0$ such that
\begin{equation}\label{06n}
\begin{split}
\norm{\rho(t_0+t)}_{L^\infty(\mathbb{R}^2)}
\leq \frac{M}{4\pi t}+C_0 M^{\frac 1p+\frac 1r} \int_0^t  (t-s)^{-\frac{1}{\varrho}-\frac{1}{2}}\norm{\rho(t_0+s)}_{L^\infty(\mathbb{R}^2)}^{2-\frac 1p-\frac 1r}ds.
\end{split}
\end{equation}
We recall the exponents present above satisfy
$
\frac 1p+\frac 1q=\frac 1\varrho (\varrho>2)$ and $\frac 1r-\frac 1q=\frac 12 (1<r<2).$
Then it can be checked that $\frac{1}{\varrho}+\frac{1}{2}=\frac 1p+\frac 1r$, hence $\frac 1p+\frac 1r<1$ since $\varrho>2$.
Taking $t_0=t$, and then multiplying the inequality (\ref{06n}) by $2t$, we get
\begin{equation}\label{ine1}
\begin{split}
&2t\norm{\rho(2t)}_{L^\infty(\mathbb{R}^2)}-\frac{M}{2\pi }\\[2mm]
&\leq 2C_0 M^\xi t\int_0^t  (t-s)^{-\xi}(t+s)^{\xi-2} \left[(t+s)\norm{\rho(t+s)}_{L^\infty(\mathbb{R}^2)}\right]^{2-\xi}ds,
\end{split}
\end{equation}
where $\xi=\frac 1p+\frac 1r$. Simple calculation gives us that
\begin{equation*}
  t\int_0^t  (t-s)^{-\xi}(t+s)^{\xi-2} ds=\frac{1}{2(1-\xi)}, \quad 0<\xi<1.
\end{equation*}
For any $t>0$, we observe that
\begin{equation*}
 \sup_{0\le s\le t} (t+s)\norm{\rho(t+s)}_{L^\infty(\mathbb{R}^2)}\le  \sup_{0\le s\le t}2s\norm{\rho(2s)}_{L^\infty(\mathbb{R}^2)}=\varphi(t).
\end{equation*}
Since $\rho\in C^0(\mathbb{R}^+; L^\infty(\mathbb{R}^2))$, $\varphi$ is continuous, then it follows from (\ref{ine1}) that
\begin{equation}\label{ine2}
 \varphi(t)\le \frac{M}{2\pi}+C_1M^\xi \varphi^\theta(t),
\end{equation}
where $\displaystyle C_1=\frac{C_0}{1-\xi}$, $\displaystyle\theta=2-\xi$.
Since $0<\xi<1$, then $1<\theta<2$.


Once we have inequality (\ref{ine2}) with $\theta>1$ and $0<\xi<1$, we can apply Lemma \ref{trap} to find a constant $C>0$ such that $\varphi(t) \leq C$ for any $t>0$ if $M$ is small. This gives us that $\|\rho(t)\|_{L^\infty(\R^2)} \leq C t^{-1}$. Then by interpolation we can get
$$\|\rho(t)\|_{L^p(\R^2)} \leq Ct^{-(1-\frac{1}{p})}.$$


\subsubsection{Higher dimensional case ($d>2$)}

Setting $\phi(|\nabla S|)\nabla S=U$ and multiplying the first equation of \eqref{system} by $p\rho^{p-1}$ $(p>1)$, we have
\begin{equation}\label{1n}
\begin{split}
\frac{d}{dt}\int_{\mathbb{R}^d}\rho^pdx+\frac{4(p-1)}{p}\int_{\mathbb{R}^d}|\nabla \rho^{\frac{p}{2}}|^2 \,dx& = p(p-1)\int_{\mathbb{R}^d} \rho^{p-1} U\cdot\nabla \rho \, dx\\
& = 2(p-1)\int_{\mathbb{R}^d} \rho^{\frac p2} U\cdot \nabla \rho^{\frac p2} \, dx.
\end{split}
\end{equation}
For the term on the right hand side of \eqref{system}, we employ Cauchy-Schwartz inequality to get
\begin{equation*}
\begin{split}
2(p-1)\int_{\mathbb{R}^d} \rho^{\frac p2} U\cdot\nabla \rho^{\frac p2} \,dx
\le \frac{p-1}{p}\int_{\mathbb{R}^d}|\nabla \rho^{\frac{p}{2}}|^2dx+p(p-1)\int_{\mathbb{R}^d} \rho^p|U|^2dx,
\end{split}
\end{equation*}
which updates \eqref{1n} as
\begin{equation}\label{2}
\begin{split}
\frac{d}{dt}\int_{\mathbb{R}^d}\rho^pdx+\frac{3(p-1)}{p}\int_{\mathbb{R}^d}|\nabla \rho^{\frac{p}{2}}|^2 \,dx \le & p(p-1)\int_{\mathbb{R}^d} \rho^p|U|^2\,dx\\
\le & p(p-1) A_0^2\int_{\mathbb{R}^d} \rho^p|\nabla S|^2\,dx,
\end{split}
\end{equation}
where we have used \eqref{asF} for the last inequality.
With the Sobolev inequality
$$
\norm{\rho}_{L^{\frac{dp}{d-p}}(\mathbb{R}^d)}\leq C_1(d,p) \norm{\nabla \rho}_{L^{p}(\mathbb{R}^d)},$$
for some constant $C_1(d,p)>0$,
we have that
\begin{equation*}
\begin{split}
\norm{\rho}^{\frac p2}_{L^{\frac{dp}{d-2}}(\mathbb{R}^d)}=&\left[\left(\int_{\mathbb{R}^d}\rho^{\frac{dp}{d-2}}dx\right)^{\frac{d-2}{dp}}\right]^{\frac p2}= \left[\int_{\mathbb{R}^d}(\rho^{\frac p2})^{\frac{2d}{d-2}}dx\right]^{\frac {d-2}{2d}}\\[2mm]
=& \norm{\rho^{\frac p2}}_{L^{\frac{2d}{d-2}}(\mathbb{R}^d)}\le C_1(d,p) \norm{\nabla\rho^{\frac p2}}_{L^2(\mathbb{R}^d)}.
\end{split}
\end{equation*}
This together with \eqref{2} gives
\begin{equation}\label{3}
\begin{split}
\frac{d}{dt}\int_{\mathbb{R}^d}\rho^pdx+C_2(d,p)\norm{\rho}^{p}_{L^{\frac{dp}{d-2}}(\mathbb{R}^d)}
\le p(p-1) A_0^2\int_{\mathbb{R}^d} \rho^p|\nabla S|^2dx.
\end{split}
\end{equation}
By H\"older inequality, we have
\begin{equation*}
\begin{split}
\int_{\mathbb{R}^d}\rho^p|\nabla S|^2dx\leq  \left[\int_{\mathbb{R}^d} (\rho^p)^{\frac{d}{d-2}}dx\right]^{\frac{d-2}{d}}\left(\int_{\mathbb{R}^d} (|\nabla S|^2)^\frac d2dx\right)^{\frac 2d}
\leq \norm{\rho}^{p}_{L^{\frac{dp}{d-2}}(\mathbb{R}^d)}\norm{\nabla S}_{L^d(\mathbb{R}^d)}^2,
\end{split}
\end{equation*}
which updates \eqref{3} as
\begin{equation}\label{4}
\begin{split}
\frac{d}{dt}\int_{\mathbb{R}^d}\rho^pdx+C_2(d,p)\norm{\rho}^{p}_{L^{\frac{dp}{d-2}}(\mathbb{R}^d)}
\le p(p-1) A_0^2 \norm{\rho}^{p}_{L^{\frac{dp}{d-2}}(\mathbb{R}^d)}\norm{\nabla S}_{L^d(\mathbb{R}^d)}.
\end{split}
\end{equation}
Now we apply the inequality (\ref{HLIN}) with $q=d$ (i.e. $\frac{d{p}}{d-{p}}=d$ and hence ${p}=\frac d2$) and get
\begin{equation*}
\begin{split}
\norm{\nabla S}_{L^d(\mathbb{R}^d)}\leq& c(d,{p})\norm{\rho}_{L^{\frac d2}(\mathbb{R}^d)}\leq C(d,{p}) \norm{\rho}_{L^1(\mathbb{R}^d)}^{\frac 2d}\norm{\rho}_{L^\infty(\mathbb{R}^d)}^{1-\frac 2d}
\leq C_3(d) M^{\frac 2d}.
\end{split}
\end{equation*}
Substituting above inequality into \eqref{4}, one has
\begin{equation}\label{5}
\begin{split}
\frac{d}{dt}\int_{\mathbb{R}^d}\rho^pdx+C_2(d,p)\norm{\rho}^{p}_{L^{\frac{dp}{d-2}}(\mathbb{R}^d)}
\le C_4(d,p)p(p-1)A_0^2M^{\frac 4d} \norm{\rho}^{p}_{L^{\frac{dp}{d-2}}(\mathbb{R}^d)},
\end{split}
\end{equation}
where $C_4(d,p) = C_3(d)p(p-1)$. Let $B=C_2(d,p)-C_4(d,p)A_0^2M^{\frac 4d}>0$ for small $M>0$, it follows from (\ref{5}) that
\begin{equation}\label{6}
\begin{split}
\frac{d}{dt}\int_{\mathbb{R}^d}\rho^pdx+B\norm{\rho}^{p}_{L^{\frac{dp}{d-2}}(\mathbb{R}^d)}
\le 0.
\end{split}
\end{equation}
By interpolation inequality, we know
\begin{equation*}
\begin{split}
\norm{\rho}_{L^p(\mathbb{R}^d)}\le \norm{\rho}_{L^1(\mathbb{R}^d)}^{\frac{2}{d(p-1)+2}}\ \norm{\rho}_{L^{\frac{dp}{d-2}}(\mathbb{R}^d)}^{\frac{d(p-1)}{d(p-1)+2}}=M^{\frac{2}{d(p-1)+2}}\norm{\rho}_{L^{\frac{dp}{d-2}}(\mathbb{R}^d)}^{\frac{d(p-1)}{d(p-1)+2}},
\end{split}
\end{equation*}
which implies
\begin{equation*}
\begin{split}
\left(\int_{\mathbb{R}^d}\rho^p dx\right)^{\frac{d(p-1)+2}{d(p-1)}}\ M^{-\frac{2p}{d(p-1)}}\le \norm{\rho}_{L^{\frac{dp}{d-2}}(\mathbb{R}^d)}^p.
\end{split}
\end{equation*}
This together with \eqref{6} gives
\begin{equation*}
\begin{split}
\frac{d}{dt}\int_{\mathbb{R}^d}\rho^pdx+B M^{-\frac{2p}{d(p-1)}}\, \left(\int_{\mathbb{R}^d}\rho^p dx\right)^{\frac{d(p-1)+2}{d(p-1)}}
\le 0.
\end{split}
\end{equation*}
Then solving above ODE gives
\begin{equation*}
\begin{split}
\norm{\rho}_{L^p(\mathbb{R}^d)} \le C(d,p,M)(1+t)^{-\frac{d(p-1)}{2p}}\le C(d,p,M)(1+t)^{-\frac{d}{2}(1-\frac 1p)}, \quad 1<p\leq\infty,
\end{split}
\end{equation*}
for some constant $C(d,p,M)>0$.

\section{Conclusion and open questions}

Motivated both by the properties of solutions as traveling pulses and some specific derivations from kinetic models,  we have considered the flux-limited Keller-Segel (FLKS) system \eqref{ksf} on the whole domain $\R^d$. First, we have introduced a new generic rescaling which allows a systematic derivation as  the limit of kinetic systems describing the chemotactic motion thanks to a run-and-tumble process.
Then, since solutions exist globally in time, we have investigated the long time asymptotic of FLKS and shown that when the degradation coefficient $\alpha>0$, diffusion takes the advantage over attraction and solutions decays in time with the same rate as solutions to the heat equation.
When $\alpha=0$, we investigated radially symmetric steady states and established that the total mass to the system $M$ is an important parameter. Indeed, in dimension $d=2$, radial symmetric solutions exist if and only if $M>\frac{8\pi}{\phi(0)}$. In dimension $d>2$ there is no positive radial steady state with finite mass.

However, we have been able to prove the long time convergence towards the radial steady state only in the particular case of dimension $d=1$. Then, we leave open the question of convergence of solutions to FLKS in dimension $d>1$. In particular, we established in Theorem \ref{thm} that, for $d=2,3$, $\alpha=0$, and $\tau=0$, solutions decay in time when $M$ is small enough. An interesting issue is to prove that the critical mass for this behaviour in dimension $d=2$ is given by $M^*=\frac{8\pi}{\phi(0)}$. We also leave  as an open question the case $\tau=1$ and $\alpha=0$, for which our approach may not be applied.

Finally, in a recent work \cite{PTV}, it has been proved that kinetic system for chemotaxis may be derived from a more elaborated system at mesoscopic scale including internal variables describing for instance the methylation level within cells. Then we may expect that the FLKS system may be derived directly from such system. The proof of such derivation is also an interesting continuation of this work.

\appendix

\section{Technical lemma}
\renewcommand{\theequation}{\thesection.\arabic{equation}}

For the sake of completeness, we present in this appendix some usefull technical estimations on the parabolic/elliptic equation satisfied by the chemical concentration $S$ in system \eqref{ksf}. We recall the notation of the heat kernel
$$
G(t,x)=\frac{1}{(4\pi t)^{d/2}}\exp\left(-\frac{|x|^2}{4t}\right), \quad x \in \R^d, \ t>0.
$$
With simple calculations, we verify the following estimates,
\begin{equation}\label{hk}
\|\partial_t^l\partial_x^kG(t)\|_{L^p(\R^d)}\leq C t^{-\frac{d}{2}(1-\frac{1}{p})-l-\frac{k}{2}}, \quad 1\leq q \leq q\leq \infty.
\end{equation}
Define
$$
(e^{t\Delta} f)(x)=G(t)\ast f(x)=\int_{\R^d} G(t,x-y)f(y)dy.$$
Then using (\ref{hk}) and Young's convolution inequality, the following $L^p$-$L^q$ estimates for the operator $e^{t\Delta}$ can be easily proved.
\begin{lemma}\label{Gheat}
Let $1\leq q \leq p\leq \infty$ and $f \in L^p(\R^d)$. Then
\begin{enumerate}
\item $\|e^{t\Delta} f\|_{L^p(\R^d)}\leq C t^{-\frac{d}{2}(\frac{1}{q}-\frac{1}{p})}\|f\|_{L^q(\R^d)}$;
\item $\|\nabla \big(e^{t\Delta} f\big)\|_{L^p(\R^d)}\leq C t^{-\frac{d}{2}(\frac{1}{q}-\frac{1}{p})-\frac{1}{2}}\|f\|_{L^q(\R^d)}$.
\end{enumerate}
\end{lemma}
Then, we state the following estimates~:
\begin{lemma}\label{g}
Let $\rho$ be given and let $S$ be a solution, for $\tau\in\{0,1\}$ and $\alpha\geq 0$, to
$$
\tau \partial_t S - \Delta S + \alpha S = \rho, \qquad S(0,x) = S^0(x) \mbox{ if } \tau = 1.
$$
Then there is a constant $C>0$ such that the following hold:
\begin{enumerate}
\item When $\alpha>0$, then
\begin{eqnarray}\label{grad1}
\begin{aligned}
&\|\nabla S(t)\|_{L^p(\R^d)}\leq C \|\rho(t)\|_{L^q(\R^d)}, \quad \mathrm{for}\ \tau=0;\\
&\|\nabla S(t)\|_{L^p(\R^d)}\leq e^{-\alpha t}\|\nabla S^0\|_{L^p(\R^d)}+C \Gamma(\beta) \sup\limits_{0<s<t} \|\rho(s)\|_{L^q(\R^d)}, \quad \mathrm{for}\ \tau=1;
\end{aligned}
\end{eqnarray}
where $1\leq q\leq p\leq \infty, \frac{1}{q}<\frac{1}{p}+\frac{1}{d}$ and $\beta=\frac{1}{2}-\frac{d}{2}(\frac{1}{q}-\frac{1}{p})>0$ and $\Gamma$ is the gamma function.
\item When $\alpha=0$, then
\begin{eqnarray}\label{grad2}
\|\nabla S(t)\|_{L^p(\R^d)}\leq C \|\rho(t)\|_{L^q(\R^d)}, \quad \frac{1}{q}=\frac{1}{p}+\frac{1}{d}, \ \mathrm{for}\ \tau=0;
\end{eqnarray}
and
\begin{eqnarray}\label{grad3}
\|\nabla S(t)\|_{L^p(\R^d)}\leq \|\nabla S^0\|_{L^p(\R^d)}+C t^{\beta} \sup\limits_{0<s<t} \|\rho(s)\|_{L^q(\R^d)}, \quad \mathrm{for}\ \tau=1;
\end{eqnarray}
where $1\leq q\leq p\leq \infty, \frac{1}{q}<\frac{1}{p}+\frac{1}{d}$ and $\beta=\frac{1}{2}-\frac{d}{2}(\frac{1}{q}-\frac{1}{p})>0$.
\end{enumerate}
\end{lemma}
\begin{proof}
(1) We first consider the case $\alpha>0$. When $\tau=0$, the solution $S$ can be written as $S(x)=\mathcal{K}_\alpha \ast \rho(x)$ where $\mathcal{K}_\alpha$ is the Bessel potential:
$$
\mathcal{K}_\alpha(x)=\frac{1}{4\pi\alpha} \int_0^\infty \exp\Big(-\frac{\pi \alpha |x|^2}{4s}-\frac{s}{4\pi}\Big)s^{1-\frac{d}{2}}\frac{ds}{s},
$$
which satisfies the following property (e.g. see \cite{ENV})
$$
\|\nabla \mathcal{K}_\alpha\|_{L^\gamma(\R^d)} \leq C, \quad 1\leq \gamma <\frac{d}{d-1},
$$
for some constant $C>0$. Then by the Young's convolution inequality, we have
$$
\|\nabla S\|_{L^p(\R^d)}=\|\nabla \mathcal{K}_\alpha \ast \rho\|_{L^p(\R^d)}\leq \|\nabla \mathcal{K}_\alpha\|_{L^\gamma(\R^d)} \|\rho\|_{L^q(\R^d)} \leq C \|\rho\|_{L^q(\R^d)},
$$
where $\frac{1}{q}=\frac{1}{p}+1-\frac{1}{\gamma}<\frac{1}{p}+\frac{1}{d}$. This gives the first inequality of (\ref{grad1}).

When $\tau=1$, we use the Duhamel's principle to write
$$
S(t)=e^{-\alpha t}e^{t\Delta}S^0+\int_0^t e^{-\alpha (t-s)} e^{(t-s)\Delta} \rho(s)ds,
$$
which gives rise thanks to Lemma \ref{Gheat} to
\begin{eqnarray}\label{ss}
\begin{aligned}
\|\nabla S(t)\|_{L^p(\R^d)}&\leq e^{-\alpha t}\|\nabla S^0\|_{L^p(\R^d)}+C\int_0^t e^{-\alpha (t-s)}(t-s)^{-\frac d2 (\frac 1q-\frac 1p)-\frac 12}\|\rho(s)\|_{L^q(\R^d)}ds\\
& \leq e^{-\alpha t} \|\nabla S^0\|_{L^p(\R^d)}+C\sup\limits_{0<s<t} \|\rho(s)\|_{L^q(\R^d)} \int_0^\infty e^{-\alpha s}s^{\beta-1}ds,
\end{aligned}
\end{eqnarray}
which yields the second inequality of \eqref{grad1}.

(2) Next we prove the case $\alpha=0$. When $\tau=0$, then from the second equation of (\ref{ksf}), we have
$$S(t,x)=\int_{\R^d} \mathcal{K}_0(x-y)\rho(t,y)dt,$$
where $\mathcal{K}_0(x)$ is the Poisson kernel given by
\begin{eqnarray*}
\mathcal{K}_0(x)=
\begin{cases}
-\frac{1}{2\pi}\log|x|, \ & d=2,\\
\frac{1}{d(d-2)\gamma(d)}\cdot \frac{1}{|x|^{d-2}},\ & d\geq 3,
\end{cases}
\end{eqnarray*}
with $\gamma(d)=\frac{\pi^{d/2}}{\Gamma(\frac{d}{2}+1)}$ denoting the volume of unit ball in $\R^d$. Then it can be easily checked that
\begin{equation}\label{gradsn}
\nabla S(t,x)=\lambda_d \int_{\R^d} \frac{x-y}{|x-y|^d}\rho(t,y)dy=\lambda_d \frac{x}{|x|^d}\ast \rho,
\end{equation}
for $\lambda_d$ a positive constant.
Recall the Hardy-Littlewood-Sobolev inequality:
Let $1<p<q<\infty$ with $\frac{1}{p}=\frac{1}{q}+\frac{\sigma}{d}$. There exists a constant $C=C(d, \sigma, p)$ such that for $f \in L^p(\R^d)$,  then
\begin{equation}\label{HLS}
\bigg\|\int_{\R^d} \frac{f(y)}{|x-y|^{d-\sigma}}dy\bigg\|_{L^q(\R^d)}\leq C \|f\|_{L^p(\R^d)}.
\end{equation}
Then applying the Hardy-Littlewood-Sobolev  inequality (\ref{HLS}) with $\sigma=1$ to (\ref{gradsn}), we get the following estimate:
\begin{equation}\label{HLIN}
\norm{\nabla S}_{L^q(\mathbb{R}^d)}\leq C(d,p) \norm{\rho}_{L^{{p}}(\mathbb{R}^d)},\ q=\frac{d{p}}{d-{p}},
\end{equation}
which gives the first inequality of (\ref{grad2}).

The second inequality of (\ref{grad2}) results from (\ref{ss}) directly by letting $\alpha=0$. Thus the proof is completed.
\end{proof}

\begin{lemma}[Singular Gronwall's inequality \cite{NSU}]\label{gronwall}
Suppose $T>0$, $b\geq 0$ and $\beta>0$. Let $a(t)$ and $f(t)$ be two nonnegative functions locally integrable on $0\leq t<T<\infty$ with
$$f(t)\leq a(t)+b\int_0^t(t-s)^{\beta-1}f(s)ds, \quad 0 \leq t <T.$$
Then there is a constant $C_\beta$ depending on $\beta$ such that
$$f(t)\leq a(t)+b\Gamma(\beta) C_\beta \int_0^t (t-s)^{\beta-1}a(s)ds, \quad 0\leq t<T.$$
\end{lemma}

{\bf Acknowledgements.}
Part of this work was done while B.P. and N.V. were visitors in Hong Kong Polytechnic University. They are grateful for the welcome of this institution.
NV. acknowledges partial funding from the ANR blanche project Kibord ANR-13-BS01-0004 funded by the French Ministry of Research. Z. Wang acknowledges an internal grant  No. 4-ZZHY of the Hong Kong Polytechnic University and Hong Kong RGC GRF grant  	PolyU 153031/17P.

\bibliography{Bibrefn}

\begin{thebibliography}{10}

\bibitem{ALT80}
W.~Alt.
\newblock Biased random walk model for chemotaxis and related diffusion
  approximation.
\newblock {\em J. Math.\ Biol.}, 9:147--177, 1980.

\bibitem{Am2}
H.~Amann.
\newblock Dynamic theory of quasilinear parabolic equations iii: Global
  existence.
\newblock {\em Math.\ Z.}, 202:219--250, 1989.

\bibitem{Am1}
H.~Amann.
\newblock Dynamic theory of quasilinear parabolic equations ii:
  Reaction-diffusion systems.
\newblock {\em Differential \ Integral \ Equations.}, 3:13--75, 1990.

\bibitem{BDEF}
A.~Blanchet, J.~Dolbeault, M.~Escobedo, and J.~Fern\'{a}ndez.
\newblock Asymptotic behavior for small mass in the two-dimensional
  parabolic-elliptic keller-segel model.
\newblock {\em J. Math. Anal. Appl.}, 361:533--542, 2010.

\bibitem{BCGP}
N.~Bournaveas, V.~Calvez, S.~Guti\`errez, and B.~Perthame.
\newblock Global existence for a kinetic model of chemotaxis via dispersion and
  strichartz estimates.
\newblock {\em Comm. P.D.E.}, 33:79--95, 2008.

\bibitem{Berg91}
E.~Budrene and H.~Berg.
\newblock Complex patterns formed by motile cells of {E}scherichia coli.
\newblock {\em Nature}, 349:630--633, 1991.

\bibitem{CalCar}
V.~Calvez and J.~A. Carrillo.
\newblock Refined asymptotics for the subcritical {K}eller-{S}egel system and
  related functional inequalities.
\newblock {\em Proc. Amer. Math. Soc.}, 140(10):3515--3530, 2012.

\bibitem{CPY}
V.~Calvez, B.~Perthame, and S.~Yasuda.
\newblock {Traveling wave and aggregation in Flux-Limited Keller-Segel model}.
\newblock to appear in Kin. Rel. Models, 2018.

\bibitem{CMPS}
F.~Chalub, P.A. Markowich, B.~Perthame, and C.~Schmeiser.
\newblock Kinetic models for chemotaxis and their drift-diffusion limits.
\newblock {\em Monatsh. Math.}, 142:123--141, 2004.

\bibitem{ChapLogas}
M.~Chaplain and G.~Logas.
\newblock Mathematical modelling of cancer cell invasion of tissue: the role of
  the urokinase plasminogen activation system.
\newblock {\em Math. \ Models \ Methods \ Appl. \ Sci.}, 15:1685--1734, 2005.

\bibitem{Chertock}
A.~Chertock, A.~Kurganov, X.~Wang, and Y.~Wu.
\newblock On a chemotaxis model with saturated chemotactic flux.
\newblock {\em Kinetic and Related Models.}, 5(1):51--95, 2012.

\bibitem{DS}
Y.~Dolak and C.~Schmeiser.
\newblock Kinetic models for chemotaxis: Hydrodynamic limits and
  spatio-temporal mechanisms.
\newblock {\em J. Math. Biol.}, 51:595--615, 2005.

\bibitem{PLoS2species}
C.~Emako, C.~Gayrard, A.~Buguin, L.~Neves~de Almeida, and N.~Vauchelet.
\newblock Traveling pulses for a two-species chemotaxis model.
\newblock {\em PLoS Comput. Biol.}, 12:e1004843, 2016.

\bibitem{ENV}
C.~Emako, L.~Neves De~Almeida, and N.~Vauchelet.
\newblock Existence and diffusive limit of a two-species kinetic model of
  chemotaxis.
\newblock {\em Kinetic and Related Models}, 8:359--380, 2015.

\bibitem{ErbanOth1}
R.~Erban and H.~Othmer.
\newblock From individual to collective behavior in bacterial chemotaxis.
\newblock {\em SIAM J. Appl. Math}, 65:361--391, 2004/2005.

\bibitem{ErbanOth2}
R.~Erban and H.~Othmer.
\newblock From signal transduction to spatial pattern formation in {E}. coli: a
  paradigm for multiscale modeling in biology.
\newblock {\em Multiscale Model. Simul.}, 3:362--394, 2005.

\bibitem{Gamba03}
A.~Gamba, D.~Ambrosi, A.~Coniglio, A.~de~Candia, S.~Di~Talia, E.~Giraudo,
  G.~Serini, L.~Preziosi, and F.~Bussolino.
\newblock Percolation, morphogenesis, and burgers dynamics in blood vessels.
\newblock {\em Phys. Rev. Lett.}, 66:11801, 2003.

\bibitem{hiloth}
T.~Hillen and H.~Othmer.
\newblock The diffusion limit of transport equations derived from velocity jump
  processes.
\newblock {\em SIAM J.\ Appl.\ Math.}, 61(3):751--775, 2000.

\bibitem{HPS}
T.~Hillen, K.~Painter, and C.~Schmeiser.
\newblock Global existence for chemotaxis with finite sampling radius.
\newblock {\em Discrete Contin. Dyn. Syst. Ser. B}, 7:125--144, 2007.

\bibitem{HP09}
T.~Hillen and K.~J. Painter.
\newblock A user's guide to pde models for chemotaxis.
\newblock {\em J. Math. Biol.}, 58:183--217, 2009.

\bibitem{HKS1}
H.J. Hwang, K.~Kang, and A.~Stevens.
\newblock Global solutions of nonlinear transport equations for chemosensitive
  movement.
\newblock {\em SIAM J. \ Math. \ Anal.}, 36:1177--1199, 2005.

\bibitem{HKS3}
H.J. Hwang, K.~Kang, and A.~Stevens.
\newblock Global existence of classical solutions for a hyperbolic chemotaxis
  model and its parabolic limit.
\newblock {\em Indiana Univ.\ Math.\ J.}, 55:289--316, 2006.

\bibitem{JV}
F.~James and N.~Vauchelet.
\newblock Chemotaxis : from kinetic equations to aggregate dynamics.
\newblock {\em Nonlinear Diff. Eq. Appl.}, 20(1):101--127, 2013.

\bibitem{KS1}
E.F. Keller and L.A. Segel.
\newblock Initiation of slime mold aggregation viewed as an instability.
\newblock {\em J.\ Theor. \ Biol.}, 26:399--415, 1970.

\bibitem{KS3}
E.F. Keller and L.A. Segel.
\newblock Traveling bands of chemotactic bacteria: a theorectical analysis.
\newblock {\em J. Theor. Biol.}, 26:235--248, 1971.

\bibitem{LiMuneoka}
S.G. Li and K.~Muneoka.
\newblock Cell migration and chick limb development: chemotactic action of
  {FGF}-4 and the {AER}.
\newblock {\em Dev. Cell.}, 211:335--347, 1999.

\bibitem{Liao}
J.~Liao.
\newblock Global solution for a kinetic chemotaxis model with internal dynamics
  and its fast adaptation limit.
\newblock {\em J. Differential Equations}, 259:6432--6458, 2015.

\bibitem{LNT}
C.-S. Lin, W.-M. Ni, and I.~Takagi.
\newblock Large amplitude stationary solutions to a chemotaxis system.
\newblock {\em J. Differential Equations}, 72:1--27, 1988.

\bibitem{Murray1}
J.D. Murray.
\newblock {\em Mathematical Biology I: An Introduction}.
\newblock Springer, Berlin, 3rd edition, 2002.

\bibitem{NSU}
T.~Nagai, R.~Syukuinn, and M.~Umesako.
\newblock Decay properties and asymptotic profiles of bounded solutions to a
  parabolic system of chemotaxis in $\mathbb{R}^n$.
\newblock {\em Funkc. Ekvac.}, 46:383--407, 2003.

\bibitem{Ni}
W.-M. Ni.
\newblock Diffusion, cross-diffusion, and their spike-layer steady states.
\newblock {\em Notices Amer. Math. Soc.}, 45:9--18, 1998.

\bibitem{ODA88}
H.~Othmer, S.R. Dunbar, and W.~Alt.
\newblock Models of dispersal in biological systems.
\newblock {\em J. Math.\ Biol.}, 26:263--298, 1988.

\bibitem{othhil}
H.~Othmer and T.~Hillen.
\newblock The diffusion limit of transport equations {II}: Chemotaxis
  equations.
\newblock {\em SIAM J.\ Appl.\ Math.}, 62(4):1122--1250, 2002.

\bibitem{PainterFish}
K.J. Painter, P.K. Maini, and H.G. Othmer.
\newblock Stripe formation in juvenile pomacanthus explained by a generalized
  {T}uring mechanism with chemotaxis.
\newblock {\em Proc. Natl. Acad. Sci.}, 96:5549--5554, 1999.

\bibitem{patlak}
C.S. Patlak.
\newblock Random walk with persistence and external bias.
\newblock {\em Bull.\ Math.\ Biophys.}, 15:311--338, 1953.

\bibitem{Perthamebook}
B.~Perthame.
\newblock {\em Transport {E}quations in {B}iology}.
\newblock Birkh\"{a}user Verlag, Basel, 2007.

\bibitem{PTV}
B.~Perthame, M.~Tang, and N.~Vauchelet.
\newblock Derivation of the bacterial run-and-tumble kinetic equation from a
  model with biochemical pathway.
\newblock {\em J. Math. Biol.}, 73:1161--1178, 2016.

\bibitem{PY}
B.~Perthame and S.~Yasuda.
\newblock {Self-organized pattern formation of run-and-tumble chemotactic
  bacteria: Instability analysis of a kinetic chemotaxis model}.
\newblock working paper or preprint, March 2017.

\bibitem{Pettet96}
G.J. Petter, H.M. Byrne, D.L.S. Mcelwain, and J.~Norbury.
\newblock A model of wound healing and angiogenesis in soft tissue.
\newblock {\em Math. Biosci.}, 136(1):35--63, 2003.

\bibitem{SCB-PLOS}
J.~Saragosti, V.~Calvez, N.~Bournaveas, A.~Buguin, P.~Silberzan, and
  B.~Perthame.
\newblock Mathematical description of bacterial traveling pulses.
\newblock {\em PLoS Computational Biology}, 6:e1000890, 2010.

\bibitem{SCB-PNAS}
J.~Saragosti, V.~Calvez, N.~Bournaveas, B.~Perthame, A.~Buguin, and
  P~Silberzan.
\newblock Directional persistence of chemotactic bacteria in a traveling
  concentration wave.
\newblock {\em PNAS}, 108:16235--16240, 2011.

\bibitem{STY}
G.~Si, M.~Tang, and X.~Yang.
\newblock A pathway-based mean-field model for {E}. coli chemotaxis:
  Mathematical derivation and its hyperbolic and parabolic limits.
\newblock {\em Multiscale Model. Simul.}, 12(2):907926, 2014.

\bibitem{SWW}
B.~D. Sleeman, M.~J. Ward, and J.~C. Wei.
\newblock The existence and stability of spike patterns in a chemotaxis model.
\newblock {\em SIAM J. Appl. Math.}, 65:790--817, 2005.

\bibitem{Wang}
X.~Wang.
\newblock Qualitative behavior of solutions of chemotactic diffusion systems:
  Effects of motility and chemotaxis and dynamics.
\newblock {\em SIAM J. Math. Anal.}, 31:535--560, 2000.

\bibitem{Xue1}
C.~Xue.
\newblock Multiscale models of taxis-driven patterning in bacterial
  populations.
\newblock {\em SIAM J. Appl. Math}, 70:133--167, 2009.

\bibitem{Xue2}
C.~Xue.
\newblock Macroscopic equations for bacterial chemotaxis: integration of
  detailed biochemistry of cell signaling.
\newblock {\em J. Math. Biol.}, 70:1--44, 2015.

\end{thebibliography}
\bibliographystyle{plain}
\end{document}